\documentclass[12pt,A4]{article}

\usepackage{amsmath,amsfonts,amssymb,amsthm}
\usepackage{mathrsfs}
\usepackage{latexsym}
\usepackage[english]{babel}
\usepackage{graphicx}
\usepackage[usenames,dvipsnames]{xcolor}

\newcommand{\orange}{\color{black}}

\renewenvironment{proof}[1][\proofname]{{\bfseries #1.} }{\qed}

\def\Cov{{\rm Cov\,}}

\newcommand{\field}[1]{\mathbb{#1}}

\newcommand{\R}{\field{R}}

\newcommand{\N}{\field{N}}

\newcommand{\e}{{\rm e}}

\newcommand{\eps}{\varepsilon}

\newcommand{\var}{\operatorname{Var}}
\newcommand{\area}{\operatorname{area}}

\newcommand{\hB}{b}

\def\authors#1{{ \begin{center} #1 \vspace{0pt} \end{center} } \smallskip}
\def\institution#1{{\sl \begin{center} #1 \vspace{0pt} \end{center} } }

\def\title#1{{\huge\bf  \begin{center} #1 \vspace{0pt} \end{center}  } \smallskip}

\newcommand{\tod}{\stackrel{d}{\longrightarrow}}

\def\E{{\mathbb{ E}}}

\def\P{{\mathbb{P}}}

\def\paref#1{(\ref{#1})}

\newtheorem{theorem}{Theorem}[section]
\newtheorem{proposition}[theorem]{Proposition}
\newtheorem{lemma}[theorem]{Lemma}

\newtheorem{defn}[theorem]{Definition}

\newtheorem{remark}[theorem]{Remark}

\begin{document}

\title{Nodal Statistics of \\ Planar Random Waves}

\date{July 27, 2017}

\authors{\large Ivan Nourdin, Giovanni Peccati and Maurizia Rossi}

\institution{\large Unit\'e de Recherche en Math\'ematiques, Universit\'e du Luxembourg}

\begin{abstract}

{We consider Berry's random planar wave model (1977) for a positive Laplace eigenvalue $E>0$, both in the real and complex case, and prove limit theorems for the nodal statistics associated with a smooth compact domain, in the high-energy limit ($E\to \infty$). Our main result is that both the nodal length (real case) and the number of nodal intersections (complex case) verify a Central Limit Theorem, which is in sharp contrast with the non-Gaussian behaviour observed for real and complex arithmetic random waves on the flat $2$-torus, see Marinucci {\it et al.} (2016) and Dalmao {\it et al.} (2016). Our findings can be naturally reformulated in terms of the nodal statistics of a single random wave restricted to a compact domain diverging to the whole plane. As such, they can be fruitfully combined with the recent results by Canzani and Hanin (2016), in order to show that, at any point of isotropic scaling and for energy levels diverging sufficently fast, the nodal length of any Gaussian pullback monochromatic wave verifies a central limit theorem with the same scaling as Berry's model. 
As a remarkable byproduct of our analysis, we rigorously confirm the asymptotic behaviour for the variances of the nodal length and of the number of nodal intersections of isotropic random waves, as derived in Berry (2002).}

\medskip


\medskip


\end{abstract}


\section{Introduction}

{\orange The aim of the present paper is to prove second order asymptotic results, in the high-energy limit, for the nodal statistics associated with the restriction of the (real and complex) {\it Berry's random wave model} \cite{berry2002} to a smooth compact domain of $\R^2$. Our main result is a Central Limit Theorem (CLT) for both quantities (see Theorems \ref{thlength} and \ref{thpoints}), yielding as a by-product a rigorous and self-contained explanation of the {\it cancellation phenomena} for the variance asymptotics of nodal lengths and nodal intersections first detected in \cite{berry2002}; this complements in particular the main findings of \cite{Wig}.

As explained below, our techniques will show that the {\it cancellation phenomena} detected in \cite{berry2002} can be explained by the partial cancellation of lower order {\it Wiener-It\^o chaotic projections}. In particular, our findings represent a substantial addition to a rapidly growing line of research, focussing on the analysis of nodal quantities by means of Wiener-It\^o chaotic expansions and associated techniques --- see e.g. \cite{CMWa, CMWb, CM, DNPR, MPRW, MRW, MW, PR, RoW}. The central limit results proved in this paper are in sharp contrast with the non-central and non-universal limit theorems established in \cite{DNPR, MPRW} for arithmetic random waves on the flat $2$-torus, and mirror the CLTs for random spherical harmonics established in \cite{MRW}. To the best of our knowledge, our findings represent the first high-energy central limit theorems for nodal quantities associated with random Laplace eigenfunctions defined on the subset of a {\it non-compact} manifold. 

As discussed in Section \ref{RRW}, our results can be naturally reformulated in terms of the nodal length and the nodal intersections of a single random wave, restricted to a compact window increasing to the whole plane. As such, they can be fruitfully combined with the findings of \cite{CH}, in order to prove CLTs for the nodal length of generic {\it pullback random waves}, locally determined by Riemaniann monochromatic waves (on a general compact manifold) at a given point of {\it isotropic scaling} -- see Theorem \ref{t:rrw} below. 

Further motivations and connections with the existing literature will be discussed in the sections to follow.

\bigskip

\noindent{\bf Some conventions.} For the rest of the paper, we assume that all random objects are defined on a common probability space $(\Omega, \mathcal{F}, \P)$, with $\E$ denoting expectation with respect to $\P$. We use the symbol $\tod$ to denote convergence in distribution, and the symbol $\stackrel{a.s.}{\longrightarrow}$ to denote $\P$ -almost sure convergence. Given two positive sequences $\{a_n\}$, $\{b_n\}$, we write $a_n \sim b_n$ if $a_n/b_n \to 1$, as $n\to \infty$.

}

\subsection{Berry's Random Wave Model}

In \cite{berry77}, Berry argued that, at least for classically chaotic quantum billiards, wavefunctions {\orange in the high-energy limit} \emph{locally} look like random superpositions of {\orange independent} plane waves, having all the same wavenumber, say $k$, but different directions. According to \cite[formula (6)]{berry2002}, such a superposition has the form
\begin{equation}\label{u real}
u_{J; k}(x) := \sqrt{\frac{2}{J}} \sum_{j=1}^J \cos\left ( k x_1\cos \theta_j + k x_2 \sin \theta_j + \phi_j  \right ),\qquad J \gg 1,
\end{equation}
where $x = (x_1, x_2)\in \R^2$, and $\theta_j$ and $\phi_j$ are, respectively, random directions and random phases such that ($\theta_1, \phi_1, \dots, \theta_J, \phi_J)$ are i.i.d. uniform random variables on $[0,2\pi]$). For dynamical systems {\orange with time-reversal symmetry}, these plane waves are real, while in the absence of time-reversal symmetry, for instance when the billiard is open, they are complex functions:
\begin{equation}\label{u complex}
u^{\mathbb C}_{J; k}(x) := u_{J; k}(x) + iv_{J; k}(x),
\end{equation}
where $v_{J; k}(x)$ is given by formula (\ref{u real}) with the cosine replaced by the sine, and the random vector $(\theta_1, \phi_1, \dots, \theta_J, \phi_J)$ is defined as above; see again \cite{berry2002}, as well as the surveys \cite{D-survey, UR-survey} and the references therein. 

The sequence $\lbrace u_{J; k} \rbrace_J$ in \paref{u real} converges {\orange in the sense of finite-dimensional distributions}, as $J\to +\infty$, to 
the centered \emph{isotropic} Gaussian field $b_k = \left\lbrace b_k(x) : x\in \R^2\right\rbrace$, with covariance kernel given by
\begin{equation}\label{covk}
c^k(x,y) =c^k(x-y) : = \E\left[\hB_k(x) \hB_k(y)\right] = J_0(k \| x- y\|),\qquad x,y\in \R^2,
\end{equation}
where $J_0$ denotes the zero-order Bessel function of the first kind:
\begin{equation}\label{seriesJ0}
J_0(t) = \sum_{m=0}^{+\infty} \frac{(-1)^m}{(m!)^2}\left (\frac{t}{2}\right)^{2m},\qquad t\in \R.
\end{equation} 
Recall that $J_0$ is the only radial solution of the equation 
\begin{equation*}
\Delta f + f= 0
\end{equation*}
such that $f(0)=1$, where $\Delta:=\partial^2/\partial x_1^2 + \partial^2/\partial x_2^2$ denotes the Laplacian on the Euclidean plane. 

It is a standard fact (see e.g. \cite[Theorem 5.7.3]{AT}) that 
we can represent $\hB_k$ as a random series 
\begin{equation}\label{serie}
\hB_k(x) = b_k(r,\theta)=\Re \left( \sum_{m=-\infty}^{+\infty} a_m J_{|m|}(kr)\e^{im\theta} \right),
\end{equation}
using polar coordinates $(r,\theta)=x$, where $\Re$ denotes the real part, $a_m$ are i.i.d. complex Gaussian random variables such that $\E[a_m]=0$ and $\E[|a_m|^2]=2$, and $J_\alpha$ stands for the Bessel function of the first kind of order $\alpha$. The series \paref{serie} is a.s. convergent, and uniformly convergent on any compact set, and the sum is a \emph{real analytic} function (this is due to the fact that the mapping $ \alpha\mapsto J_\alpha(z)$ is asymptotically equivalent to $\alpha^{-1/2}(2z/\pi\alpha)^{\alpha}$, as $\alpha\to +\infty$ --- see e.g. \cite[formula (9.3.1)]{AS}). From \paref{serie} it follows also that $\hB_k$ is a.s. an \emph{eigenfunction} of the Laplacian $\Delta$ on $\R^2$ with eigenvalue $-k^2$, i.e., $b_k$ solves the Helmholtz equation
$$
\Delta \hB_k(x) + k^2 \hB_k(x) = 0, \quad x\in \R^2.
$$
{\orange A standard application e.g. of \cite[Theorem 5.7.2]{AT} also shows the following reverse statement: if $Y$ is an isotropic centered Gaussian field on the plane, with unit variance and such that $\Delta Y+k^2Y = 0$, then necessarily $Y$ has the same distribution as $\hB_k$. This also shows that, for every $k>0$, the two Gaussian random functions $x\mapsto b_k(x)$ and $x\mapsto b_1(kx)$ have the same distribution.} 

The `universal' random field $\hB_k$ is known as {\it Berry's Random Wave Model}, and is the main object of our paper. The complex version of $\hB_k$ is the random field 
\begin{equation}\label{e:complexberry}
\hB^{\mathbb C}_k(x) := \hB_k(x) + i \widehat{\hB}_k(x), \quad x\in \R^2,
\end{equation}
 where $\widehat{\hB}_k$ is an independent copy of $\hB_k$. {\orange We observe that $\hB^{\mathbb C}_k$ can be represented as a random series as well, and that such a representation is obtained by removing the symbol $\Re$ on the right-hand side of \paref{serie}}. It follows in particular that $\hB^{\mathbb C}_k$ a.s. verifies the equation $\Delta \hB^{\mathbb C}_k +k^2\hB^{\mathbb C}_k=0$, that is, $\hB^{\mathbb C}_k$ is a.s. a complex-valued solution of the Helmholtz equation associated with the eigenvalue $ - k^2$.

\subsection{Mean and variance of nodal statistics (Berry, 2002)}\label{intro-nodal}

The principal focus of our analysis are the two \emph{nodal sets}:
$$
\hB_k^{-1}(0) := \lbrace x\in \R^2 : \hB_k(x) = 0\rbrace, \,\, \mbox{and}\,\,  (\hB^{\mathbb C}_k)^{-1}(0) =
\hB_k^{-1}(0)\cap (\widehat{\hB}_k)^{-1}(0).$$
It is proved in Lemma \ref{lemnodal} of Appendix A that $\hB_k^{-1}(0)$ is a.s. {\orange a union of} smooth curves (called {\it nodal lines}), while $(\hB^{\mathbb C}_k)^{-1}(0)$ is a.s. composed of isolated points (often referred to as {\it phase singularities} or {\it optical vortices} -- see \cite{D-survey, UR-survey}).

In \cite{berry2002}, the distributions of the length ${\bf l}_k$ of the nodal lines of $b_k$ and of the number ${\bf n}_k$ of nodal points of its complex version, when restricted to some fixed {billiard} $\mathcal{D}$, were studied. For the means of the latter quantities, Berry found that
\begin{equation}\label{berry mean}
\begin{split}
\E[ {\bf l}_k ] = \frac{Ak }{2\sqrt{2}}\quad \mbox{and} \quad
\E[ {\bf n}_k ] = \frac{Ak^2}{4\pi},
\end{split}
\end{equation} 
where $A$ denotes the area of $\mathcal{D}$,
 while for their high-energy fluctuations, some semi-rigorous computations led to the following asymptotic relations, valid as $k\to \infty$:
\begin{equation}\label{berry fluct}
\begin{split}
\var({\bf l}_k)  \sim  \frac{A}{256\pi}\log(k\sqrt{A}),\,\,\, \mbox{and}\,\,\,
\var({\bf n}_k) \sim \frac{11A k^2}{64\pi^3}\log(k\sqrt{A}).
\end{split}
\end{equation}

\bigskip

{\orange According to \cite{berry2002}, the unexpected logarithmic order of both variances in \eqref{berry fluct} is due to an ``obscure cancellation phenomenon'', corresponding to an exact simplification of seveal terms appearing in the {\it Kac-Rice formula} --- see the discussion below --- as applied to the computation of variances. As anticipated, our aim in this paper is to prove a CLT both for $\mathbf l_k$ and $\mathbf n_k$, yielding as a by-product a rigorous explanation of \eqref{berry fluct} in terms of the partial cancellation of lower order Wiener-It\^o chaotic components.}

\subsection{Main results}

In order to make more transparent the connection with some relevant parts of the recent literature (see Section \ref{ss:introrelated}), for the rest of the paper we set, for $E>0$,
$$
{\orange B_E(x) := b_k(x), \quad x\in \R^2,}
$$
where $k: = 2\pi\sqrt{E}$, in such a way that the covariance of $B_E$ is given by
\begin{equation}\label{covE}
r^E(x, y) = r^E(x-y) := J_0(2\pi \sqrt E \| x - y\|), \,\,\, x,y\in \R^2;
\end{equation}
see \paref{covk}. {\orange Analogously, for $E>0$ we write
$$
B_E^{\mathbb{C}}(x) := b^{\mathbb{C}}_k(x) = B_E(x) + i \widehat{B} _E(x), \quad x\in \mathbb{R}^2,
$$
where $k = 2\pi\sqrt{E}$, and $\widehat{B} _E$ is an independent copy of $B_E$.}

\medskip

Let us now fix a \emph{$C^1$-convex body} $\mathcal D\subset \R^2$ (that is: $\mathcal{D}$ is a compact convex set with $C^1$-boundary) such that $0\in \mathring{\mathcal D}$ (i.e., the origin belongs to the interior of $\mathcal D$). 
The restriction of the zero set of $B_E$ to $\mathcal D$ is 
\begin{equation*}
B_E^{-1}(0)\cap \mathcal D = \lbrace x\in \mathcal D : B_E(x) = 0\rbrace.
\end{equation*}
According to Lemma \ref{lemnodal} in Appendix A, the set $B_E^{-1}(0)$ intersects the boundary $\partial \mathcal D$ in an a.s. finite number of points. The {\it nodal length} of $B_E$ restricted to $\mathcal D$ is the random variable
\begin{equation}\label{deflength}
\mathcal L_E:= \text{length}(B_E^{-1}(0)\cap \mathcal D),
\end{equation}
which is square-integrable, by Lemma \ref{L2-lemma} below. The first main result of the present paper concerns the distribution of $\mathcal L_E$ in the high-energy limit.
\begin{theorem}\label{thlength}
The expectation of the nodal length $\mathcal{L}_E$ is 
\begin{equation}\label{meanlength}
\E[\mathcal L_E] = \area(\mathcal D)\,\frac{\pi}{\sqrt{2}}\sqrt{E},
\end{equation}
whereas the variance of $\mathcal L_E$ verifies the asymptotic relation
\begin{equation}\label{e:varlength}
\var(\mathcal L_E) \sim \area(\mathcal D)\,\frac{1}{512\pi}\log E, \quad E\to\infty. 
\end{equation}
Moreover, as $E\to \infty$,
\begin{equation*}
\frac{\mathcal L_E - \E[\mathcal L_E]}{\sqrt{\var(\mathcal L_E)}}\mathop{\longrightarrow}^{d} Z,
\end{equation*}
where $Z\sim \mathscr{N}(0,1)$ is a standard Gaussian random variable. 
\end{theorem}

{\orange
\begin{remark}{\rm Relation \eqref{meanlength} coincides with \cite[formula (19)]{berry2002} (and (\ref{berry mean}) above), whereas (\ref{e:varlength}) is consistent with \cite[formula (28)]{berry2002} (and (\ref{berry fluct}) above). 
}
\end{remark}

\begin{remark}{\rm 
In what follows, we will use the relation 
\begin{equation}\label{e:eqdist}
\mathcal{L}_E \stackrel{d}{=} \frac{1}{2\pi \sqrt{E}} \, \text{length}\Big(b_1^{-1}(0)\cap 2\pi\sqrt{E}\cdot \mathcal D\Big),
\end{equation}
where $\stackrel{d}{=} $ indicates equality in distribution and, for $a>0$, we set $a \cdot \mathcal D := \{y\in \mathbb{R}^2 : y=ax, x\in \mathcal{D}\}$. Such an equality in distribution is an immediate consequence of the integral representation of nodal lengths appearing e.g. in \eqref{int_rapp} below, as well as of the fact that, as random functions, $b_1(2\pi \sqrt{E} x)$ and $B_E(x)$ have the same distribution for every $E>0$.

}
\end{remark}
}

We now focus on the complex Berry's RWM $B_E^{\mathbb C}$, and study the nodal points (phase singularities) of $B_E^{\mathbb C}$ that belong to a $C^1$ convex body $\mathcal D$ such as the one considered above (in particular, the origin lies in the interior of $\mathcal{D}$). As already observed, one has that 
$$
(B_E^{\mathbb C} )^{-1} (0) = B_E^{-1}(0) \cap \widehat B_E^{-1}(0),
$$ 
and the set $(B_E^{\mathbb C} )^{-1} (0)\cap \mathcal D$ consists $\mathbb{P}$-a.s. of a finite collection of points such that none of them belongs to the boundary $\partial \mathcal D$ (see Lemma \ref{lemnodal}). We are interested in the distribution of 
\begin{equation}\label{defN}
\mathcal N_E := \# \Big( (B_E^{\mathbb C})^{-1}(0)  \cap \mathcal D \Big),
\end{equation}
for large values of $E$. Our second main result is the following:
\begin{theorem}\label{thpoints}
One has that
\begin{equation}\label{e:meanps}
\E[\mathcal N_E] = \area(\mathcal D)\,\pi E.
\end{equation}
Moreover, as $E\to \infty$, 
\begin{equation} \label{e:varps}
\var(\mathcal N_E) \sim \area(\mathcal D)\,\frac{11}{32\pi}E \log E,
\end{equation}
and
\begin{equation*}
\frac{\mathcal N_E - \E[\mathcal N_E]}{\sqrt{\var(\mathcal N_E)}}\mathop{\longrightarrow}^{d} Z,
\end{equation*}
where $Z\sim \mathscr{N}(0,1)$ is a standard Gaussian random variable. 
\end{theorem}

\begin{remark}{\rm Relation \eqref{e:meanps} coincides with  \cite[(45)]{berry2002} (or (\ref{berry mean})) above) whereas \eqref{e:varps} is the same as \cite[(50)]{berry2002} (or (\ref{berry fluct}) above).
}
\end{remark}

{\orange We will now show how Theorem \ref{thlength} can be combined with the findings of \cite{CH}, in order to deduce local CLTs for {\it pullback (monochromatic) random waves} associated with a general Riemaniann manifold. }

{\orange
\subsection{Application to monochromatic random waves}\label{RRW}

\subsubsection{Random waves on manifolds}

Let $(\mathcal M, g)$ be a compact, smooth, Riemannian manifold of dimension $2$. We write $\Delta_g$ to indicate the associated Laplace-Beltrami operator, and denote by $\lbrace f_j : j\in \N \rbrace$ an orthonormal basis of $L^2(\mathcal M)$ composed of real-valued eigenfunctions of $\Delta_g$
$$
\Delta_g f_j + \lambda^2_j f_j=0,
$$
where the corresponding eigenvalues are such that  $0=\lambda_0<\lambda_1\le \lambda_2 \le \dots \uparrow \infty$. According to \cite{CH, Z}, the (Riemannian) {\it monochromatic random wave} on $M$ of parameter $\lambda$ is defined as the Gaussian random field
\begin{equation}\label{phi}
\phi_\lambda(x) := \frac{1}{\sqrt{\text{dim}(H_{c,\lambda})}} \sum_{\lambda_j\in [\lambda, \lambda + c]} a_j f_j(x), \quad x\in M,
\end{equation}
where $c\geq 0$ is a fixed parameter and the $a_j$ are i.i.d. standard Gaussian random variables, and
$$
H_{c, \lambda} := \bigoplus_{\lambda_j \in [\lambda, \lambda + c]} \text{Ker}(\Delta_g + \lambda_j^2\, \text{Id}),
$$
where $\text{Id}$ is the identity operator. The field $\phi_\lambda$ is centered Gaussian, and its covariance kernel is given by 
\begin{equation}\label{covRiem}
K_{c,\lambda}(x,y) := \Cov(\phi_\lambda(x), \phi_\lambda(y)) = \frac{1}{\text{dim}(H_{c,\lambda})} \sum_{\lambda_j\in [\lambda, \lambda +c]} f_j(x) f_j(y),\quad x,y\in \mathcal M.
\end{equation}
``Short window'' monochromatic random waves such as $\phi_\lambda$ (in the case $c=1$ and for manifolds of arbitrary dimension) were first introduced by Zelditch in \cite{Z} as general approximate models of random Gaussian Laplace eigenfunctions defined on manifolds not necessarily having spectral multiplicities; see \cite{CH} for further discussions. The case $c=0$ typically corresponds to manifolds with spectral multiplicities like the flat torus $\mathbb{R}^2/\mathbb{Z}^2$ or the round sphere $\mathbb{S}^2$, where one can consider models of random waves living on a single eigenspace (like {\it arithmetic random waves} \cite{RW}, and {\it random spherical harmonics} \cite{Wig}) -- see also the forthcoming Section \ref{ss:introrelated}. Plainly, for a generic metric on a smooth compact manifold $\mathcal M$, the eigenvalues $\lambda_j^2$ are simple, and one has to average on intervals $[\lambda, \lambda+c]$ such that $c>0$ in order to obtain a non-trivial probabilistic model.

\subsubsection{Pulback random waves and isotropic scaling}

We keep the notation introduced in the previous section, and follow closely \cite{CH}. Fix $x\in \mathcal M$, and consider the tangent plane $T_x\mathcal M$ to the manifold at $x$. We define the {\it pullback Riemannian random wave} associated with $\phi_\lambda$ as the Gaussian random field on $T_x\mathcal M$ given by 
$$
\phi_\lambda^x(u) := \phi_\lambda\left ( \exp_x \left ( \frac{u}{\lambda} \right )\right ),\qquad u\in T_x \mathcal M, 
$$
where $\exp_x : T_x\mathcal M \to \mathcal M$ is the exponential map at $x$. The planar field $\phi_\lambda^x$ is trivially centered and Gaussian and, using \paref{covRiem}, its covariance kernel is given by
$$
K_{c,\lambda}^x(u,v) = K_{c,\lambda}\left(\exp_x \left ( \frac{u}{\lambda}  \right) , \exp_x \left ( \frac{v}{\lambda}\right ) \right ),\qquad u,v\in T_x \mathcal M.  
$$ 

\begin{defn}[See \cite{CH}]{\rm We say that $x \in \mathcal{M}$ is a point of {\it isotropic scaling} if, for every positive function $\lambda \mapsto r(\lambda)$ such that $r(\lambda) = o(\lambda)$, as $\lambda\to \infty$, one has that
\begin{equation}\label{limit}
\sup_{u,v \in \mathbb{B}(r(\lambda))  } \left| \partial^\alpha\partial^\beta[  K_{c,\lambda}^x(u,v) - (2\pi) J_0(\| u-v\|_{g_x})]\right| \to 0,\quad \lambda\to \infty,
\end{equation}
where $\alpha, \beta\in \mathbb{N}^2$ are multi-indices labeling partial derivatives with respect to $u$ and $v$, respectively, $\| \cdot \|_{g_x}$ is the norm on $T_x\mathcal M$ induced by $g$, and $\mathbb{B}(r(\lambda))$ is the corresponding ball of radius $r(\lambda)$ containing the origin.
}
\end{defn}

Sufficient conditions for a point $x$ to be of isotropic scaling are discussed e.g. in \cite[Section 2.5]{CH} or \cite{CH0}. In the case $c=0$, one can directly verify that every point $x\in \mathbb{S}^2$ is of isotropic scaling for the model of random spherical harmonics evoked above (see \cite{Wig}), and a similar analysis could be implemented on the flat torus $\mathbb T^2$, but only for a density-one subsequence of Laplace eigenvalues -- see e.g. \cite{KKW, KW}. Note that one can always choose coordinates around $x$ to have $g_x = \text{Id}$, so that the limiting kernel in \paref{limit} coincides with $(2\pi)\times c^1$ in \paref{covk}. This implies in particular that, if $x$ is a point of isotropic scaling, then, as $\lambda \to \infty$, the planar field $\phi_\lambda^x$ converges to a multiple of Berry's model, namely $\sqrt{2\pi}\cdot b_1$, in the sense of finite-dimensional distributions. 

\subsubsection{A second order result}
Keeping the same notation and assumptions as above, we now state a special case of \cite[Theorem 1]{CH}, that we reformulate in a way that is adapted to the notation adopted in the present paper. To this end, for every $x\in \mathcal{M}$ we define
$$
\mathcal{Z}_{\lambda, E}^x := {\rm length}\left \{ (\phi_\lambda^x)^{-1}(0)\cap \mathbb{B}(2\pi\sqrt{E})\right\}, \quad E>0.
$$
The next statement shows that, if $x$ is of isotropic scaling, then $\mathcal{Z}_{\lambda, E}^x$ behaves, for large values of $\lambda$ as the universal random quantity given by the nodal length of Berry's model $b_1$ restricted to the ball $ \mathbb{B}(2\pi\sqrt{E})$.

\begin{theorem}[Special case of Theorem 1 in \cite{CH}]\label{t:ch}Let $x$ be a point  of isotropic scaling, and assume that coordinates have been chosen around $x$ in such a way that $g_x = {\rm Id.}$ Fix $E>0$. Then, as $\lambda \to \infty$,  the random variable $\mathcal{Z}_{\lambda, E}^x$ converges in distribution to 
$$
{\rm length}\Big(b_1^{-1}(0)\cap  \mathbb{B}(2\pi\sqrt{E}) \Big) \,\,\,\left ( \stackrel{d}{=} \mathcal{L}_E\cdot 2\pi\sqrt{E} \right ),
$$
where the identity in distribution expressed between brackets follows from \eqref{e:eqdist}.

\end{theorem}

The next statement is a direct consequence of Theorem \ref{thlength}, and provides a second-order counterpart to Theorem \ref{t:ch}, showing in particular that nodal lengths of pullback random waves inherit high-energy Gaussian fluctuations from Berry's model at any point of isotropic scaling. In order to make the statement more readable, we introduce the notation
$$
\widetilde{\mathcal{Z}}_{\lambda, E}^x :=  \frac{\mathcal{Z}_{\lambda, E}^x}{2\pi\sqrt{E}}. 
$$

\begin{theorem}[CLT for the nodal length of pullback waves]\label{t:rrw} Let $x$ be a point  of isotropic scaling, and assume that coordinates have been chosen around $x$ in such a way that $g_x = {\rm Id.}$ Let $\{E_m : m\geq 1\}$ be a sequence of positive numbers such that $E_m \to \infty$. Then, there exists a sequence $\{\lambda_m : m\geq 1\}$ such that 
\begin{equation}\label{e:pbclt}
\frac{\widetilde{\mathcal{Z}}_{\lambda_m, E_m}^x - \pi^2\sqrt{E_m/2}}{\sqrt{\log( E_m) /512} }\tod Z\sim\mathscr{N}(0,1).
\end{equation}

\end{theorem}

\noindent \begin{proof} Let ${\bf d}(\cdot, \cdot)$ be any distance metrizing the convergence in distribution between random variables (see e.g. \cite[Appendix C]{NP}), and let $\epsilon(m)$, $m\geq 1$, be a sequence of positive numbers such that $\epsilon(m)\to 0$. According to Theorem \ref{t:ch}, for every fixed $m$ there exists $\lambda_m >0$ such that 
$$
{\bf d}\left (\frac{\widetilde{\mathcal{Z}}_{\lambda_m, E_m}^x - \pi^2\sqrt{E_m/2}}{\sqrt{\log( E_m) /512} } \, , \, \frac{\mathcal{L}_{E_m} - \pi^2\sqrt{E_m/2}}{\sqrt{\log( E_m) /512} } \right) \leq \epsilon(m).
$$
From this relation we deduce that, for every $m$,
$$
{\bf d}\left (\frac{\widetilde{\mathcal{Z}}_{\lambda_m, E_m}^x - \pi^2\sqrt{E_m/2}}{\sqrt{\log( E_m) /512} }\, ,\,  Z \right) \leq \epsilon(m) + {\bf d}\left (\frac{\mathcal{L}_{E_m} - \pi^2\sqrt{E_m/2}}{\sqrt{\log( E_m) /512} }\, , \, Z \right),
$$
and the conclusion follows at once from Theorem \ref{thlength} .
\end{proof}

\medskip

It would be of course desirable to have some quantitative information about the sequence $\lambda_m$, $m\geq 1$ appearing in the previous statement, in particular connecting the asymptotic behaviour of $\lambda_m$ with the speed of divergence of $E_n$. Some preliminary computations have indicated us that (not suprisingly) in order to do so, one should have explicit upper bounds on the limiting relation \eqref{limit}, that one should exploit in order to deduce a quantitative version of Theorem \ref{t:ch}. We prefer to think of this issue as a separate problem, and leave it open for further research.

}

\subsection{Further related work}\label{ss:introrelated}

The distribution of the nodal length on the standard flat torus $\mathbb T^2$ and on the unit round sphere $\mathbb S^2$ was investigated in \cite{RW, KKW, MPRW, PR} and \cite{Ber, Wig, MRW}, respectively. Moreover, the distribution of the number of nodal points on $\mathbb T^2$ was studied in \cite{DNPR}. 
Remember that, as mentioned in \S \ref{RRW}, since these manifolds have spectral degeneracies, one typically selects the value $c= 0$ in (\ref{phi}) for defining a  canonical model of Gaussian random waves. We will now describe in more detail the theoretical contributions contained in the references evoked above. A more technical comparison with the approach adopted in the present work is deferred to Section \ref{ss:comp}.
\medskip

\noindent{\bf Nodal length of real arithmetic random waves}. The eigenvalues of the Laplace operator on $\mathbb T^2$ are of the form $-4\pi^2n$, where $n$  is an integer that can be represented as the sum of two integer squares. Write $S$ for the collection of all integers having this property, and, for $n\in S$, denote by $\Lambda_n$ the set of frequencies 
$$
\Lambda_n = \lbrace \xi\in \mathbb Z^2 : \| \xi\|= \sqrt n\rbrace 
$$
and by $\mathcal N_n$ the cardinality of $\Lambda_n$ (that is, $\mathcal N_n$ is the multiplicity of $-4\pi^2n$). For $n\in S$, consider the probability measure $\mu_n$ induced by $\Lambda_n$ on the unit circle $\mathbb S^1$:
$$
\mu_n = \frac{1}{\mathcal N_n} \sum_{\xi\in \Lambda_n} \delta_{\xi/\sqrt n}.
$$ 
Following \cite{RW}, for $n\in S$, the toral random eigenfunction $T_n$ (or {\it arithmetic random wave} of order $n$) is defined as the centered Gaussian field on the torus whose covariance function is as follows: for $x,y\in \mathbb T^2$,
\begin{equation}\label{cov_int}
\Cov(T_n(x), T_n(y)) = \frac{1}{\mathcal N_n} \sum_{\xi\in \Lambda_n} \e^{i 2\pi\langle \xi, x-y\rangle } = \int_{\mathbb S^1} \e^{i2\pi \sqrt{n} \langle \theta, x-y\rangle}\,d\mu_n(\theta).
\end{equation}
As discussed in $\cite{KKW} $, there exists a density-$1$ subsequence $\lbrace n_j : j\geq 1\rbrace\in S$ such that, as $j\to +\infty$, 
$$
\mu_{n_j} \Rightarrow d\theta/2\pi,
$$
where $d\theta$ denotes the uniform measure on the unit circle. 
Let us now set $\mathcal L_n := \text{length}(T_n^{-1}(0))$. The expected nodal length was computed by Rudnick and Wigman \cite{RW}:
$$
\E[\mathcal L_n] = \frac{1}{2\sqrt 2} \sqrt{4\pi^2n},
$$
while in \cite{KKW} the asymptotic variance, as $\mathcal N_n\to +\infty$, was proved to be
$$
\var(\mathcal L_n) \sim \frac{1+\widehat{\mu_n}(4)^2}{512} \, \frac{4\pi^2n}{\mathcal N_n^2},
$$
where $\widehat{\mu_n}(4)$ denotes the fourth Fourier coefficients of $\mu_n$. In order to have an asymptotic law for the variance, one should select to a subsequence $\lbrace n_j\rbrace$ of energy levels such that (i) $\mathcal N_{n_j}\to +\infty$ and (ii) $|\widehat{\mu_n}(4)|\to \eta$,  for some $\eta\in [0,1]$. Note that for each $\eta\in [0,1]$, there exists a subsequence $\lbrace n_j\rbrace$ such that both (i) and (ii) hold (see \cite{KKW, KW}). For these subsequences, the asymptotic distribution of the nodal length was shown to be non-Gaussian  in \cite{MPRW}: 
\begin{equation}\label{NCLT}
\frac{\mathcal L_{n_j} - \E[\mathcal L_{n_j}]}{\sqrt{\var(\mathcal L_{n_j})}}\mathop{\to}^d \frac{1}{2\sqrt{1+\eta^2}} (2 - (1-\eta) Z_1^2 - (1+\eta)Z_2^2),
\end{equation}
where $Z_1$ and $Z_2$ are i.i.d. standard Gaussian random variables. A complete quantitative version (in Wasserstein distance) of \paref{NCLT} is given in \cite{PR}.  Reference \cite{RoW} contains Limit Theorems for the intersection number of the nodal lines $T_n^{-1}(0)$ and a fixed deterministic curve with nowhere zero curvature.

\medskip

\noindent{\bf Phase singularities of complex arithmetic random waves}. For $n\in S$, let $\widehat T_n$ indicate an independent copy of the arithmetic random wave $T_n$ defined in the previous paragraph. In \cite{DNPR}, the distribution of the  cardinality $\mathcal I_n$ of the set of nodal intersections $T_n^{-1}(0)\cap \widehat T_n^{-1}(0)$ was investigated. One has that 
$$
\E[\mathcal I_n] = \frac{4\pi^2n}{4\pi} = \pi n,
$$
while the asymptotic variance, as $\mathcal N_n\to +\infty$, is 
$$
\var(\mathcal I_{n}) \sim \frac{3\widehat{\mu_n}(4)^2+5}{128\pi^2}\, \frac{(4\pi^2n)^2}{\mathcal N_n^2}. 
$$
Also in this case the asymptotic distribution is non-Gaussian (and non-universal), indeed for $\lbrace n_j\rbrace$ such that $\mathcal N_{n_j}\to +\infty$ and $|\widehat{\mu_{n_j}}(4)|\to \eta\in [0,1]$, one has that
\begin{equation*}
\frac{\mathcal I_{n_j} - \E[\mathcal I_{n_j}]}{\sqrt{\var(\mathcal I_{n_j})}}\mathop{\to}^d \frac{1}{2\sqrt{10+6\eta^2}}\left( \frac{1+\eta}{2} A + \frac{1-\eta}{2} B - 2(C-2)\right),
\end{equation*}
where $A$, $B$ and $C$ are independent random variables such that $A\stackrel{d}{=} B \stackrel{d}{=} 2Z_1^2 +2Z_2^2 - 4Z_3^2$ while $C\stackrel{d}{=} Z_1^2 + Z_2^2$ (where $Z_1, Z_2, Z_3$ are i.i.d. standard Gaussian random variables). 

\medskip

\noindent{\bf Nodal length of random spherical harmonics}. The Laplacian eigenvalues on the two-dimensional unit sphere are of the form $-\ell(\ell+1)$, where $\ell\in \mathbb N$, and the multiplicity of the $\ell$-th eigenvalue is $2\ell+1$. The $\ell$-th random eigenfunction (random spherical harmonic) on $\mathbb S^2$ is a centered Gaussian field whose covariance kernel is 
$$
\Cov(T_\ell(x)), T_\ell(y)) = P_\ell(\cos d(x,y)),\qquad x,y\in \mathbb S^2,
$$
where $P_\ell$ denotes the $\ell$-th Legendre polynomial and $d(x,y)$ the geodesic distance between the two points $x$ and $y$ (see \cite{MPbook}). 
The mean of the nodal length $\mathcal L_\ell := \text{length}(T_\ell^{-1}(0))$ was computed in \cite{Ber} as
$$
\E[\mathcal L_\ell] = \frac{1}{2\sqrt 2} \sqrt{\ell(\ell+1)},
$$
while the asymptotic behaviour of the variance was derived in \cite{Wig}: as $\ell\to +\infty$, 
$$
\var(\mathcal L_\ell) \sim \frac{1}{32}\log \ell.
$$
The second order fluctuations of $\mathcal L_\ell$ are Gaussian;  more precisely, in \cite{MRW} it was shown that 
$$
\frac{\mathcal L_\ell - \E[\mathcal L_\ell]}{\sqrt{\var(\mathcal L_\ell)}}\mathop{\to}^d Z,
$$
where $Z$ is a standard Gaussian random variable.

\section{Outline of the paper}

\subsection{On the proofs of the main results}

A well-known consequence of the area/co-area formulae and of the fact that $B_E$ is $\mathbb{P}$-a.s. a smooth field, is that one can represent in integral form the nodal length $\mathcal L_E$ in \paref{deflength} and the number of nodal points $\mathcal N_E$ in \paref{defN}, respectively, as 
\begin{eqnarray}\label{int_rapp}
\mathcal L_E &=& \int_{\mathcal D} \delta_0(B_E(x)) \| \nabla B_E(x)\|\,dx,\\
 \mathcal N_E &=& \int_{\mathcal D} \delta_0(B_E(x))\delta_0(\widehat B_E(x)) |\text{Jac}_{B_E, \widehat B_E},(x)| \,dx,\label{int_rapp1}
\end{eqnarray}
where $\delta_0$ denotes the Dirac mass at $0$, $\nabla B_E$ is the gradient field, and $\text{Jac}_{B_E, \widehat B_E}$  stands for the Jacobian of $(B_E, \widehat B_E)$ (remember that $\widehat B_E$ is an independent copy of $B_E$); on the right-hand sides of \eqref{int_rapp} and \eqref{int_rapp1}, integrals involving Dirac masses have to be understood as $\mathbb{P}$-a.s. limits of analogous integrals, where $\delta_0$ is replaced by an adequate approximation of the identity. 
We will show in Section \ref{s:msq} that $\mathcal L_E$ and  $\mathcal N_E$ are both square-integrable random variables.  Combined with \eqref{int_rapp} and \eqref{int_rapp1}, this will allow us to deploy in Section \ref{sec-chaos} the powerful theory of {\it Wiener-It\^o chaos expansions} (see e.g. \cite{NP}), yielding that both $\mathcal L_E$ and $\mathcal N_E$ admit an explicit representation as orthogonal  series, both converging in $L^2(\P)$, with the form
\begin{equation}\label{series_intro}
\mathcal L_E = \sum_{q=0}^{+\infty} \mathcal L_E[2q],\qquad \qquad \mathcal N_E = \sum_{q=0}^{+\infty} \mathcal N_E[2q],
\end{equation}
where $\mathcal L_E[2q]$ (resp. $\mathcal N_E[2q]$) denotes the orthogonal projection of $\mathcal L_E$ (resp. $\mathcal N_E$) 
onto the $2q$th {\it Wiener chaos} associated with $B_E$ (and $\widehat{B}_E$) --- see Section \ref{sec-chaos} and \cite{NP} for definitions and further details. We will see that chaotic decompositions rely in particular on the fact that the sequence of renormalized Hermite polynomials $\lbrace H_q/ \sqrt{q!} \rbrace_{q=0,1,\dots}$ is an orthonormal basis for the space of square-integrable functions on the real line w.r.t. the standard Gaussian density. Note that odd chaoses in \paref{series_intro} vanish, since the integrands on the right-hand sides of \paref{int_rapp} and \paref{int_rapp1} are even. 

Our main argument for proving Theorem \ref{thlength} and Theorem \ref{thpoints}  relies on the investigation of those chaotic components  in \paref{series_intro} such that $q\geq 1$ (the $0$-th chaotic component is the mean). The second chaotic components ($q=1$) is investigated in Section \ref{sec_2}, where we use the first Green's identity in order to show that $\mathcal L_E[2]$ and $\mathcal N_E[2]$ both reduce to a single boundary term, yielding that 
\begin{equation}\label{joy1}
\var(\mathcal L_E[2]) = O(1),\qquad \var(\mathcal N_E[2]) = O\left ( E\right ).
\end{equation}
The (more difficult) investigation of fourth chaotic components is carried out in Section \ref{sec_4}: it requires in particular a careful analysis of asymptotic moments of Bessel functions on growing domains, see Section \ref{sec_bessel}. Our main finding from Section \ref{sec_4} is that 
\begin{equation}\label{joy2}
\var(\mathcal L_E[4]) \sim \area(\mathcal D)\,\frac{1}{512\pi}\log E,\qquad \var(\mathcal N_E[4]) \sim \area(\mathcal D)\,\frac{11}{32\pi}E \log E.
\end{equation}
In Section \ref{sec_high}, we will show that the contribution of higher order chaotic components is negligible, that is: as $E\to +\infty$, 
\begin{equation}\label{joy3}
\var\left (\sum_{q\ge 3} \mathcal L_E[2q]   \right ) = o(\log E),\qquad \var \left ( \sum_{q\ge 3} \mathcal N_E[2q]   \right )=o(E \log E).
\end{equation}
This is done by exploiting isotropic property of the field, and by using a Kac-Rice formula to control the second moments of $\mathcal{L}_E$ and $\mathcal{N}_E$ around the origin. 

Substituting \paref{joy1}, \paref{joy2} and \paref{joy3} into \paref{series_intro}, we deduce that the variance of the fourth chaotic component of $\mathcal{L}_E$ and  $\mathcal{N}_E$ is asymptotically equivalent to the corresponding total variances, more precisely: as $E\to +\infty$, 
\begin{equation}\label{intro_4}
\frac{\mathcal L_E - \E[\mathcal L_E]}{\sqrt{\var(\mathcal L_E)}} = \frac{\mathcal L_E[4]}{\sqrt{\var(\mathcal L_E[4])}} + o_{\P}(1), \qquad \frac{\mathcal N_E - \E[\mathcal N_E]}{\sqrt{\var(\mathcal N_E)}} = \frac{\mathcal N_E[4]}{\sqrt{\var(\mathcal N_E[4])}} + o_{\P}(1),
\end{equation}
where $o_\P(1)$ denotes a sequence converging to zero in probability. Both relations appearing in \paref{intro_4}, indicate that, in order to conclude the proofs Theorem \ref{thlength} and Theorem \ref{thpoints}, it is sufficient to check that the normalized projections
$$
\frac{\mathcal L_E[4]}{\sqrt{\var(\mathcal L_E[4])}} \quad \mbox{and} \quad \frac{\mathcal N_E[4]}{\sqrt{\var(\mathcal N_E[4])}} 
$$
have asymptotically Gaussian fluctuations. Exploiting the fact that both quantities live in a fixed Wiener chaos, this task will be accomplished in Section \ref{sec_clt}, by using techniques of Gaussian analysis taken from \cite[Chapter 5 and 6]{NP}, in particular related to the {\it fourth moment theorem} from \cite{NuPe, PT05}.

\subsection{Further comparison with previous work}\label{ss:comp}

The idea of proving limit theorems for \emph{nodal} quantities of random Laplace eigenfunctions, by first deriving the chaos decompositions \paref{series_intro} and then by proving that the fourth chaotic projection is dominating, first appeared in \cite{MPRW}, and has been further developed in the already quoted references \cite{DNPR, MRW, PR,  RoW}. While the techniques adopted in the present paper are directly connected to such a line of research, several crucial differences with previous contributions should be highlighted.

\begin{itemize}

\item[(i)] Differently from \cite{MPRW, DNPR, MRW, PR}, the random fields considered in the present paper are eigenfunctions of the Laplace operator of a {\it non-compact} manifold (namely, the plane), that one subsequently restricts to a smooth compact domain $\mathcal D$. This situation implies in particular that, throughout our proofs and differently from \cite{DNPR, MPRW, MRW, PR}, we cannot exploit any meaningful representation of $B_E$ (or $B^{\mathbb{C}}_E$) in terms of a countable orthogonal basis of Laplace eigenfunctions on $\mathcal{D}$, thus making our computations considerably more delicate. In particular, the representation \eqref{serie} cannot be directly used in our framework.  This additional difficulty explains, in particular, the need of developing novel estimates for Bessel functions on growing domains, as derived in Section \ref{sec_bessel}.

\item[(ii)] Another consequence of the non-compactness of $\R^2$ is that (differently from the situation in \cite{MPRW, DNPR, PR}) it is not possible to  represent the dominating chaotic projections $\mathcal{L}_E[4]$ and $\mathcal{N}_E[4]$ as an explicit functional of a finite collection of independent Gaussian coefficients. This imply in particular that, in order to show that $\mathcal{L}_E[4]$ and $\mathcal{N}_E[4]$ exhibit Gaussian fluctuations, one cannot rely on the usual CLT, but one has rather to apply the analytical techniques based on the use of {\it contractions} described in \cite[Chapter 5]{NP} --- see Section \ref{sec_clt}. 

\item[(iii)] Differently from \cite{MPRW, MRW}, our proof of the variance asymptotic behaviour for nodal quantities \paref{e:varlength} and \paref{e:varps} is done from scratch, and does not make use of previous computations in the literature. In particular, our analysis provides a self-contained rigorous proof of Berry's relations \paref{berry fluct}.
\end{itemize}

\subsection{Plan}

In Section \ref{sec_chaos} we derive the chaotic decomposition \paref{series_intro} for the nodal length and the number of nodal points. 
The second chaotic components are investigated in Section \ref{sec_2} to obtain \paref{joy1}, whereas the main results on asymptotic moments of Bessel functions are in Section \ref{sec_bessel} (further technical results are collected in Appendix B). The fourth chaotic components are studied in Section \ref{sec_4} in order to obtain \paref{joy2}, and \paref{joy3} is proven in Section \ref{sec_high}. 
The Central Limit Theorem for the fourth chaotic component is proved in Section \ref{sec_clt}. Finally, the proof of our main results is given in Section \ref{sec_proof}. Additional technical lemmas are gathered together in Appendix A and Appendix C. 

\subsection*{Acknowledgements}

The research leading to this work was supported by the Grant F1R-MTH-PUL-15CONF--CONFLUENT (Ivan Nourdin), by the Grant F1R-MTH-PUL-15STAR-STARS (Giovanni Peccati and Maurizia Rossi) at the University of Luxembourg and by the FNR grant O17/11756789/FoRGES (Giovanni Peccati). 

\section{Nodal statistics and Wiener chaos}\label{sec_chaos}

\subsection{Mean square approximation}\label{s:msq}

In order to derive the chaotic decomposition \paref{series_intro} for the nodal length and the number of nodal points, we will need the distribution of the random vector 

\noindent $(B_E(x), B_E(y), \nabla B_E(x), \nabla B_E(y))\in \R^6$ for $x, y\in \R^2$, where $\nabla B_E$ is the gradient field $\nabla := (\partial_1, \partial_2), \partial_i := \partial_{x_i} = \partial/\partial {x_i}$ for $i=1,2$). 
Let us introduce the following notation: for $i,j\in \lbrace 0,1,2 \rbrace$ 
\begin{equation}\label{notationDerivative}
r^E_{i,j}(x-y) := \partial_{x_i} \partial_{y_j} r^E(x-y),
\end{equation}
with $\partial_{x_0}$ and $\partial_{y_0}$ equal to the identity by definition.
The following result will be proved in Appendix A.
\begin{lemma}\label{lemsmooth}
The centered Gaussian vector $(B_E(x), B_E(y), \nabla B_E(x), \nabla B_E(y))\in \R^6$ ($x\ne y\in \R^2$) has the following covariance matrix:
\begin{equation}\label{sigma}
\Sigma^E (x-y)= \begin{pmatrix}
\Sigma_{1}^E(x-y) &\Sigma_{2}^E(x-y)\\
\Sigma_{2}^E(x-y)^t &\Sigma_{3}^E(x-y)
\end{pmatrix},
\end{equation}
where 
$$
\Sigma_1^E(x-y) = \begin{pmatrix}
1 &r^E(x-y)\\
r^E(x-y) &1
\end{pmatrix},
$$
$r^E$ being defined in \paref{covE}, 
\begin{equation}\label{matrixB}
\Sigma_2^E(x-y) = \begin{pmatrix}
0 &0 &r_{0,1}^E(x-y) &r_{0,2}^E(x-y)\\
-r_{0,1}^E(x-y) &-r_{0,2}^E(x-y) &0 &0
\end{pmatrix},
\end{equation}
with, for $i=1,2$,
\begin{equation*}
r_{0,i}^E(x-y) =
2\pi\sqrt{E} \,\frac{x_i-y_i}{\|x-y\|}\,  J_1(2\pi\sqrt{E}\|x-y\|).
\end{equation*}
Finally
$$
\Sigma_3^E(x-y) = \begin{pmatrix}
2\pi^2E  &0  &r^E_{1,1}(x-y) &r^E_{1,2}(x-y)\\
0 &2\pi^2 E &r^E_{2,1}(x-y) &r^E_{2,2}(x-y)\\
r^E_{1,1}(x-y) &r^E_{2,1}(x-y) &2\pi^2E &0\\
r^E_{1,2}(x-y) &r^E_{2,2}(x-y) &0 &2\pi^2E
\end{pmatrix},
$$
where for $i=1,2$
\begin{equation}\label{covii}
r^E_{i,i}(x-y)=  2\pi^2 E \left ( J_0(2\pi\sqrt{E}\|x-y\|) 
 + \Big (1 - 2\frac{(x_i - y_i)^2}{\|x-y\|^2}  \Big ) J_2(2\pi\sqrt{E}\|x-y\|) \right),
\end{equation}
and 
\begin{equation}\label{cov12}
r_{12}^E(x-y) = r^E_{2,1}(x-y)= -4\pi^2 E \frac{(x_1 - y_1)(x_2 - y_2)}{\|x - y\|^2} J_2(2\pi\sqrt{E}\|x-y\|).
\end{equation}
\end{lemma}
For brevity, we will sometimes omit the dependence of $x-y$ in the covariance matrix \paref{sigma} just above, as well as in \paref{notationDerivative}. 
In view of Lemma \ref{lemsmooth}, we define the normalized derivatives as 
\begin{equation}\label{normalized}
\widetilde \partial_i := \frac{\partial_i}{\sqrt{2\pi^2E}},\qquad i=1,2,
\end{equation}
and accordingly the normalized gradient $\widetilde \nabla$ as 
\begin{equation}\label{normgrad}
\widetilde \nabla := (\widetilde \partial_1, \widetilde \partial_2) = \frac{\nabla}{\sqrt{2\pi^2E}}.
\end{equation}
Let us now consider, for $\eps>0$, the following random variables
\begin{eqnarray}\label{eps-approx}
\mathcal L_E^\eps &:=& \frac{1}{2\eps}\int_{\mathcal D} 1_{[-\eps, \eps]}(B_E(x))\| \nabla B_E(x)\|\,dx,\\
\label{eps-approxN}
\mathcal N_E^\eps &:=& \frac{1}{(2\eps)^2}\int_{\mathcal D} 1_{[-\eps, \eps]}(B_E(x))1_{[-\eps, \eps]}(\widehat B_E(x))\left | \text{Jac}_{B_E, \widehat B_E}(x)\right |\,dx,
\end{eqnarray}
where $\text{Jac}_{B_E, \widehat B_E}$ still denotes the Jacobian of $(B_E, \widehat B_E)$. The random objects in \paref{eps-approx} and \paref{eps-approxN} 
can be viewed as $\eps$-approximations of the nodal length of $B_E$ in $\mathcal D$ and of the number of nodal points  of $B_E^{\mathbb C}$ in $\mathcal D$, respectively (here and in what follows, $1_{[-\eps, \eps]}$ denotes the indicator functions of the interval $[-\eps, \eps]$). Indeed, the following standard result holds, which will be proved in Appendix A for completeness. 
\begin{lemma}\label{lemma-as-conv}
As $\eps \to 0$, 
\begin{equation}\label{as-conv}
\mathcal L_E^\eps \mathop{\longrightarrow}^{a.s.} \mathcal L_E,
\end{equation}
where $\mathcal L^\eps_E$ (resp. $\mathcal L_E$) is given in \paref{eps-approx} (resp. \paref{deflength}). Moreover 
\begin{equation}\label{as-convN}
\mathcal N_E^\eps \mathop{\longrightarrow}^{a.s.} \mathcal N_E,
\end{equation}
where $\mathcal N^\eps_E$ (resp. $\mathcal N_E$) is given in \paref{eps-approxN} (resp. \paref{defN}).
\end{lemma}
The next lemma, also proved in Appendix A, shows that the convergence in Lemma \ref{lemma-as-conv} holds in $L^2(\P)$. 
\begin{lemma}\label{L2-lemma}
The nodal length $\mathcal L_E$ in \paref{deflength} and the number of nodal points $\mathcal N_E$ \paref{defN} are finite-variance random variables, and both convergences in \paref{as-conv} and \paref{as-convN} hold in $L^2(\P)$, i.e.: as $\eps\to 0$, 
\begin{equation}\label{2conv}
\E[|\mathcal L_E^\eps - \mathcal L_E|^2]\rightarrow 0,
\end{equation}
\begin{equation}\label{2convN}
\E[|\mathcal N_E^\eps - \mathcal N_E|^2]\rightarrow 0.
\end{equation}
\end{lemma}

\subsection{Chaotic expansions} \label{sec-chaos}

The field $B_E$ can be expressed in terms of Wiener-It\^o  integral as 
\begin{equation}\label{isonormal}
B_E(x) = \frac{1}{\sqrt{2\pi}} \int_{\mathbb S^1} \e^{2\pi  \sqrt Ei\langle \theta, x\rangle}\, dG(\theta),\quad x\in \R^2,
\end{equation}
where $G$ is a complex Hermitian Gaussian measure with Lebesgue control measure (see \cite[\S 2.1]{NP} and in particular Example 2.1.4). Indeed, by the integral representation \cite[\S 9.1]{AS} of Bessel functions, 
\begin{equation}\label{int_rep}
\E[B_E(x)B_E(y)]=\frac{1}{2\pi} \int_{\mathbb S^1} \e^{2\pi \sqrt E i\langle \theta, x-y\rangle}\, d\theta = r^E(x-y), \quad x,y\in \R^2. 
\end{equation}

\begin{remark}\label{r:noisonormal} {\rm We will sometimes prefer to represent such quantities as $B_E(x)$, $\partial_1 B_E(x)$ and so on as stochastic integrals of deterministic kernels with respect to a real-valued Gaussian measure (and not a complex-valued one, as in \eqref{isonormal} --- this is alway possible, due to standard properties of separable real Hilbert spaces). See e.g. Section 8, where such a representation is implicitly used for dealing with contraction operators.}
\end{remark}

The random variables $\mathcal L_E^\eps$ and  $\mathcal N_E^\eps$ having finite variance (Lemma \ref{L2-lemma}) functionals of $B_E$ in \paref{isonormal}, they admit the
so-called \emph{chaotic expansion} \cite[\S 2.2]{NP}, i.e. they can be written as a random \emph{orthogonal} series
\begin{equation}\label{chaosexp}
\mathcal L_E^\eps = \sum_{q=0}^{+\infty} \mathcal L_E^\eps[q],\qquad \mathcal N_E^\eps = \sum_{q=0}^{+\infty} \mathcal N_E^\eps[q],
\end{equation} 
converging in $L^2$. The term $\mathcal L_E^\eps[q]$ (resp. $\mathcal N_E^\eps[q]$) 
is the orthogonal projection of $\mathcal L^\eps_E$ (resp. $\mathcal N_E^\eps$) onto the so-called $q$th Wiener chaos $C_q$ \cite[Definition 2.2.3]{NP}. The definition of the latter involves the sequence of Hermite polynomials $\lbrace H_n\rbrace_{n\ge 0}$ \cite[Definition 1.4.1]{NP} which are a complete orthonormal basis (up to normalization) of the space of square integrable functions on the real line w.r.t. the standard Gaussian density. We recall here the expression of the first Hermite polynomials:
\begin{equation}\label{hermite}
H_0(t) = 1,\  H_1(t)=t,\  H_2(t) = t^2 -1,\  H_3(t) = t^3 - 3t,\  H_4(t) = t^4 - 6t^2 +3.
\end{equation}
We recall also that for normalized $Z_1, Z_2$ jointly Gaussian, we have for any $n,n' \in \lbrace 0,1,2,\dots\rbrace$
\begin{equation}\label{herm}
\E[H_n(Z_1) H_{n'}( Z_2)]= \delta_n^{n'} n! \, \E[Z_1 Z_2]^n. 
\end{equation}
In view of \paref{herm} and Lemma \ref{lemsmooth}, we rewrite \paref{eps-approx} and \paref{eps-approxN} as
\begin{eqnarray}\label{eps-approxNorm}
\mathcal L_E^\eps &=& \frac{\sqrt{2\pi^2E}}{2\eps}\int_{\mathcal D} 1_{[-\eps, \eps]}(B_E(x))\| \widetilde \nabla B_E(x)\|\,dx,\\
\label{eps-approxNNorm}
\mathcal N_E^\eps &=& \frac{2\pi^2 E}{(2\eps)^2}\int_{\mathcal D} 1_{[-\eps, \eps]}(B_E(x))1_{[-\eps, \eps]}(\widehat B_E(x))\left | \widetilde{\text{Jac}}_{B_E, \widehat B_E}(x)\right |\,dx,
\end{eqnarray}
where $\widetilde \nabla$ is the normalized gradient \paref{normgrad}, and $\widetilde{\text{Jac}}_{B_E, \widehat B_E}$ denotes the Jacobian of $(B_E, \widehat B_E)$ w.r.t. the normalized derivatives \paref{normalized}. 
 
The chaotic expansion for $\mathcal L_E^\eps$ (resp. $\mathcal N_E^\eps$) can be obtained as in \cite[Lemma 3.4, Lemma 3.5]{MPRW} (resp. as in the proof of {\cite[Lemma 4.4]{DNPR}}) (the terms corresponding to odd chaoses vanish, due to the parity of integrand functions in \paref{eps-approx} and \paref{eps-approxN}). The proof of the following result is hence omitted. 
\begin{lemma}\label{eps-chaos}
 The chaotic components of $\mathcal L_E^\eps$ in \paref{eps-approxNorm} corresponding to odd chaoses vanish, i.e. 
\begin{equation*}
\mathcal L_E^\eps[2q+1] = 0,\qquad q\ge 0,
\end{equation*}
while for even chaoses 
\begin{equation*}
\begin{split}
\mathcal L_E^\eps[2q] = \sqrt{2\pi^2E} & \sum_{u=0}^{q} \sum_{m=0}^{u} \beta_{2q - 2u}^\eps \alpha_{2m, 2u - 2m}\times \cr
&\times \int_{\mathcal D} H_{2q-2u}(B_E(x)) H_{2m}(\widetilde \partial_1 B_E(x)) H_{2u-2m}(\widetilde \partial_2 B_E(x))\,dx,
\end{split}
\end{equation*}
where $\lbrace \beta_{2n}^\eps\rbrace_{n\ge 0}$ is the sequence of chaotic coefficients of $\frac{1}{2\eps} 1_{[-\eps, \eps]}$ appearing in \cite[Lemma 3.4]{MPRW}, while $\lbrace \alpha_{2n,2m}\rbrace_{n,m\ge 0}$ is the sequence of chaotic coeffients of the Euclidean norm in $\R^2$ $\| \cdot \|$ appearing in \cite[Lemma 3.5]{MPRW}.

The chaotic components of $\mathcal N_E^\eps$ in \paref{eps-approxNNorm} are
\begin{equation*}
\mathcal N_E^\eps[2q+1] = 0,\qquad q\ge 0,
\end{equation*}
while for even chaoses 
\begin{equation*}
\begin{split}
&\mathcal N_E^\eps[2q] = 2\pi^2E  \sum_{i_1+i_2+i_3+j_1+j_2+j_3=q} \beta_{i_1}^\eps \beta_{j_1}^\eps\gamma_{i_2,i_3,j_2,j_3}\times \cr
&\times \int_{\mathcal D} H_{i_1}(B_E(x)) H_{i_1}(\widehat B_E(x))  H_{i_2}(\widetilde \partial_1 B_E(x)) H_{i_3}(\widetilde \partial_2 B_E(x))H_{i_2}(\widetilde \partial_1 \widehat B_E(x)) H_{i_3}(\widetilde \partial_2 \widehat B_E(x))\,dx,
\end{split}
\end{equation*}
where $i_1, j_1$ are even, and $i_2, i_3, j_2, j_3$ have the same parity; here the sequence $\lbrace \gamma_{i_2,i_3,j_2,j_3}\rbrace$ corresponds to the chaotic expansion of the absolute value of the Jacobian appearing in \cite[Lemma 4.2]{DNPR}. 
\end{lemma}
Let us define, as in \cite[Lemma 3.4]{MPRW}, 
\begin{equation}\label{chaosdirac}
\beta_{2n} := \lim_{\eps \to 0} \beta^\eps_{2n}.
\end{equation}
The sequence $\lbrace \beta_{2n}\rbrace_{n\ge 0}$ consists of the (formal) chaotic coefficients of the Dirac mass $\delta_0$. Hence from Lemma \ref{L2-lemma} and Lemma \ref{eps-chaos} we immediately obtain the chaotic expansions  for $\mathcal L_E$ and $\mathcal N_E$. 
\begin{proposition} The chaotic expansion of the nodal length in $\mathcal D$ is 
\begin{equation}\label{expL}
\begin{split}
\mathcal L_E = \sum_{q=0}^{+\infty} \mathcal L_E[2q] 
=&\sqrt{2\pi^2E} \sum_{q=0}^{+\infty} \sum_{u=0}^{q} \sum_{m=0}^{u} \beta_{2q - 2u} \alpha_{2m, 2u - 2m}\times \cr
 &\times \int_{\mathcal D} H_{2q-2u}(B_E(x)) H_{2m}(\widetilde \partial_1 B_E(x)) H_{2u-2m}(\widetilde \partial_2 B_E(x))\,dx,
\end{split}
\end{equation}
where $\lbrace \beta_{2n}\rbrace_{n\ge 0}$ is defined in \paref{chaosdirac} (see also \cite[Lemma 3.4]{MPRW}), while $\lbrace \alpha_{2n,2m}\rbrace_{n,m\ge 0}$ is the sequence of chaotic coeffients of the Euclidean norm in $\R^2$ $\| \cdot \|$ appearing in \cite[Lemma 3.5]{MPRW}.

For the number of phase singularities in $\mathcal D$ we have
\begin{equation}\label{expN}
\begin{split}
&\mathcal N_E = \sum_{q=0}^{+\infty} \mathcal N_E[2q] = 2\pi^2E  \sum_{q=0}^{+\infty}  \sum_{i_1+i_2+i_3+j_1+j_2+j_3=q} \beta_{i_1} \beta_{j_1}\gamma_{i_2,i_3,j_2,j_3}\times \cr
&\times \int_{\mathcal D} H_{i_1}(B_E(x)) H_{i_1}(\widehat B_E(x))  H_{i_2}(\widetilde \partial_1 B_E(x)) H_{i_3}(\widetilde \partial_2 B_E(x))H_{i_2}(\widetilde \partial_1 \widehat B_E(x)) H_{i_3}(\widetilde \partial_2 \widehat B_E(x))\,dx,
\end{split}
\end{equation}
where $i_1, j_1$ are even, and $i_2, i_3, j_2, j_3$ have the same parity; here the sequence $\lbrace \gamma_{i_2,i_3,j_2,j_3}\rbrace$ corresponds to the chaotic expansion of the absolute value of the Jacobian appearing in \cite[Lemma 4.2]{DNPR}. 
\end{proposition}
We will need the explicit values of few chaotic coefficients for $\mathcal L_E$ and $\mathcal N_E$ (see \cite[Lemma 4.3]{DNPR} and the proofs of \cite[Proposition 3.2]{MPRW} and \cite[Lemma 4.2]{MPRW}): 
(Dirac mass) 
\begin{equation}\label{fewbeta}
\beta_0=\frac{1}{\sqrt{2\pi}},\quad
\beta_2=-\frac{1}{2\sqrt{2\pi}}, \quad  \beta_4=\frac{1}{8\sqrt{2\pi}};
\end{equation}
and 
\begin{equation}\label{fewalpha}
\begin{split}
\alpha_{0,0}=\frac{\sqrt{2\pi}}2,\quad 
\alpha_{2,0}=\alpha_{0,2}=\frac{\sqrt{2\pi}}8,\quad 
\alpha_{4,0}=\alpha_{0,4}=-\frac{\sqrt{2\pi}}{128},\quad 
\alpha_{2,2}=-\frac{\sqrt{2\pi}}{64},
\end{split}
\end{equation}
finally
\begin{equation}\label{fewgamma}
\begin{split}
\gamma_{0,0,0,0}&=1,\quad 
\gamma_{2,0,0,0}=\gamma_{0,2,0,0}=\gamma_{0,0,2,0}=\gamma_{0,0,0,2}=\frac14,\\
\gamma_{1,1,1,1}&=-\frac38,\quad 
\gamma_{2,2,0,0}=\gamma_{0,0,2,2}=-\frac1{32},\\
\gamma_{2,0,2,0}&=\gamma_{0,2,0,2}=-\frac1{32},\quad 
\gamma_{2,0,0,2}=\gamma_{0,2,2,0}=\frac{5}{32},\\
\gamma_{4,0,0,0}&=\gamma_{0,4,0,0}=\gamma_{0,0,4,0}=\gamma_{0,0,0,4}=-\frac3{192}.
\end{split}
\end{equation}

\section{Second chaotic components}\label{sec_2}

In this section we investigate the second chaotic component of the nodal length and the number of nodal components, respectively. 

\begin{lemma}
For the second chaotic component of $\mathcal L_E$ we have 
\begin{equation}\label{newsecondchaos}
\mathcal{L}_E[2]  = \frac1{8\pi\sqrt{2\,E}}
\int_{\partial\mathcal{D}} B_E(x)\langle \nabla B_E(x),n(x)\rangle dx,
\end{equation}
where $n(x)$ is the outward pointing normal at $x$, 
hence 
\begin{equation}\label{var2L}
\var(\mathcal L_E[2]) = O(1).
\end{equation}
\end{lemma}
\begin{proof}
Equation \paref{expL} implies that the projection $\mathcal{L}_E[2]$ of $\mathcal{L}_E$ onto the second chaos is given by
\begin{eqnarray}
\nonumber \mathcal{L}_E[2] &=&\sqrt{2\pi^2E} \left\{\beta_2\alpha_{0,0} \int_{\mathcal{D}}H_2(B_E(x))dx + 
\beta_0\alpha_{0,2} \int_{\mathcal{D}}H_2(\widetilde{\partial}_1 B_E(x))dx\right.\\
&&\hskip5.8cm\left.
+\beta_0\alpha_{2,0} \int_{\mathcal{D}}H_2(\widetilde{\partial}_2 B_E(x))dx\right\}
\label{secondchaos} \\
&=&\frac{\pi}{8}\sqrt{2E}
\nonumber \left\{ -2\int_{\mathcal{D}} B_E(x)^2 dx + \int_{\mathcal{D}} \|\widetilde{\nabla} B_E(x)\|^2 dx
\right\},
\end{eqnarray}
where we used the explicit expression of the second Hermite polynomial \paref{hermite}.
The first Green identity \cite[p.44]{L} (see also \cite[Proposition 7.3.1]{Ros15} and the proof of \cite[Lemma 4.4]{DNPR}) asserts that
\[
\int_{\mathcal{D}} \|\nabla B_E(x)\|^2 dx = -\int_{\mathcal{D}} B_E(x)\Delta B_E(x) dx
+\int_{\partial\mathcal{D}} B_E(x)\langle \nabla B_E(x),n(x)\rangle dx
\]
where $n(x)$ denotes the outward pointing unit normal at $x$.
As a result,
\begin{eqnarray*}
\int_{\mathcal{D}} \|\widetilde{\nabla} B_E(x)\|^2 dx
&=&\frac{1}{2\pi^2E}\int_{\mathcal{D}} \|\nabla B_E(x)\|^2 dx\\
&=&2\int_{\mathcal{D}} B_E(x)^2 dx + \frac{1}{2\pi^2 E}\int_{\partial\mathcal{D}} B_E(x)\langle \nabla B_E(x),n(x)\rangle dx,
\end{eqnarray*}
implying in turn from \paref{secondchaos} that
\begin{equation}\label{newsecondchaos2}
\mathcal{L}_E[2]  = \frac1{8\pi\sqrt{2\,E}}
\int_{\partial\mathcal{D}} B_E(x)\langle \nabla B_E(x),n(x)\rangle dx,
\end{equation}
which is \paref{newsecondchaos}. 
From \paref{newsecondchaos2} we deduce \paref{var2L}, indeed, 
\begin{eqnarray*}
{\rm Var}(\mathcal{L}_E[2])&\leq&
\frac{1}{128\pi^2 E}\int_{\partial\mathcal{D}}\E [B_E(x)^2]\, dx \cdot \int_{\partial\mathcal{D}}\E[ \|\nabla B_E(x)\|^2]\, dx\\ \nonumber
&=&\frac{1}{64}\,\mbox{perimeter}(\mathcal{D})^2 =O(1).
\end{eqnarray*}
\end{proof}

\begin{lemma}
For the second chaotic component of $\mathcal N_E$ we have 
\begin{equation}\label{newsecondchaosN}
\mathcal{N}_E[2]  = \sqrt{2E}\big(\mathcal{L}_E[2]+\widehat{\mathcal{L}}_E[2]\big)
\end{equation}
(with obvious notation), hence 
\begin{equation}\label{var2N}
\var(\mathcal N_E[2]) = O(E).
\end{equation}
\end{lemma}
\begin{proof}
Similarly to \paref{secondchaos}, from \paref{expN} we have
\begin{eqnarray*}
\mathcal{N}_E[2] &=&2\pi^2E \left\{\beta_2\beta_0\gamma_{0,0,0,0} \int_{\mathcal{D}}H_2(B_E(x))dx + 
\beta_0\beta_2\gamma_{0,0,0,0} \int_{\mathcal{D}}H_2(\widehat{B}_E(x))dx\right.\\
&&\hskip1.3cm+\beta_0^2\gamma_{2,0,0,0} \int_{\mathcal{D}}H_2(\widetilde{\partial}_1 B_E(x))dx
+\beta_0^2\gamma_{0,2,0,0} \int_{\mathcal{D}}H_2(\widetilde{\partial}_2 B_E(x))dx
\\
&&\left. \hskip1.3cm+\beta_0^2\gamma_{0,0,2,0} \int_{\mathcal{D}}H_2(\widetilde{\partial}_1 \widehat{B}_E(x))dx+
\beta_0^2\gamma_{0,0,0,2} \int_{\mathcal{D}}H_2(\widetilde{\partial}_2 \widehat{B}_E(x))dx
\right\}\\
&=&\frac{\pi E}{4}
\left\{ -2\int_{\mathcal{D}} B_E(x)^2 dx
 + \int_{\mathcal{D}} \|\widetilde{\nabla} B_E(x)\|^2 dx
-2\int_{\mathcal{D}} \widehat{B}_E(x)^2 dx
 + \int_{\mathcal{D}} \|\widetilde{\nabla} \widehat{B}_E(x)\|^2 dx
\right\}.
\end{eqnarray*}
That is,
$\mathcal{N}_E[2] = \sqrt{2E}\big(\mathcal{L}_E[2]+\widehat{\mathcal{L}}_E[2]\big)$ \paref{newsecondchaosN}, implying in turn \paref{var2N} (cf. \paref{var2L})
\[
{\rm Var}(\mathcal{N}_E[2]) =\frac{E}{16}\,\mbox{perimeter}(\mathcal{D})^2 =O(E).
\]

\end{proof}

\section{Moments of Bessel functions}\label{sec_bessel}

In order to investigate the fourth chaotic components of $\mathcal L_E$ and $\mathcal N_E$, we first need a technical result on moments of Bessel functions on convex bodies. 

Let us define (cf. \paref{notationDerivative}), for $k,l\in \lbrace 0,1,2\rbrace$, 
$$
\widetilde r^E_{k,l}(x,y) = \widetilde r^E_{k,l}(x-y) := \E \left [\widetilde \partial_k B_E(x) \widetilde \partial_l B_E(y) \right ],\qquad x,y\in \R^2,
$$
with $\widetilde \partial_0 B_E := B_E$. Note that $\widetilde r^E_{0,0}\equiv r^E$. 

Since for $n=0,1,2$, 
$$
J_n(\psi) = O\left (  \frac{1}{\sqrt \psi} \right )
$$
uniformly for $\psi\in [0, +\infty)$ (see \cite{szego}), 
from Lemma \ref{lemsmooth} we have that
for every $k,l\in \lbrace 0,1,2\rbrace$, 
\begin{equation}\label{Ogrande}
\widetilde r^E_{k,l}(x-y) = O\left(\frac{1}{\sqrt { \sqrt E \| x-y\| }} \right )
\end{equation}
uniformly, where the constant involved in the $'O'$-notation does not depend on $E$. 

Now let $(\phi, \theta)$ be standard polar coordinates on $\R^2$ ($\phi\in [0,+\infty), \theta\in [0, 2\pi]$). From Lemma \ref{lemsmooth} we have 
\begin{eqnarray*}
\widetilde r^E_{0,1}((\phi \cos \theta, \phi \sin \theta)) =
\cos \theta\,  J_1(2\pi\sqrt{E}\phi),\quad \widetilde r^E_{0,2}((\phi \cos \theta, \phi \sin \theta)) =
\sin \theta\,  J_1(2\pi\sqrt{E}\phi),
\end{eqnarray*}
and $\widetilde r^E_{i,0}=-\widetilde r^E_{0,i}$ for $i=1,2$. Moreover 
\begin{equation*}
\widetilde r^E_{1,1}((\phi \cos \theta, \phi \sin \theta))=   \left ( J_0(2\pi\sqrt{E}\phi ) 
 + \Big (1 - 2\cos^2 \theta  \Big ) J_2(2\pi\sqrt{E}\phi ) \right),
\end{equation*}
\begin{equation*}
\widetilde r^E_{2,2}((\phi \cos \theta, \phi \sin \theta))=  \left ( J_0(2\pi\sqrt{E}\phi ) 
 + \Big (1 - 2\sin^2 \theta  \Big ) J_2(2\pi\sqrt{E}\phi ) \right). 
\end{equation*}
Finally
\begin{equation*}
\widetilde r^E_{1,2}((\phi \cos \theta, \phi \sin \theta))= -2\cos \theta \cdot \sin \theta\,  J_2(2\pi\sqrt{E}\phi )=\widetilde r^E_{2,1}((\phi \cos \theta, \phi \sin \theta)).
\end{equation*}
Recall now that the diameter of $\mathcal D$ is defined as
 $$
\text{diam}(\mathcal D):= \sup_{x,y\in \mathcal D} \| x-y\|,
$$
while its inner radius is 
$$
\text{inrad}(\mathcal D) := \sup \lbrace r>0 : \exists x\in \mathcal D \text{ s.t. } B_r(x)\subseteq \mathcal D\rbrace.
$$
As briefly anticipated above, the next two propositions contain key results to investigate the asymptotic behavior of fourth order chaotic components variances in \S \ref{sec_4}, in particular for the proofs of Lemmas \ref{oddio1}--\ref{oddio2h} which are collected in Appendix B. 
\begin{proposition}\label{propbessel}
Let $q_{i,j}\ge 0$ for $i,j=0,1,2$ and $\sum_{i,j=0}^2 q_{i,j} = 4$. Then 
\begin{equation}\label{toprove}
\begin{split}
&\int_{\mathcal D} \int_{\mathcal D} \prod_{i,j=0}^2 \widetilde r^E_{i,j}(x-y)^{q_{i,j}}\,dx dy\cr
& = \text{area}(\mathcal D) \int_0^{\text{diam}(\mathcal D)}\phi\,  d\phi   \int_0^{2\pi} d\theta  \prod_{i,j=0}^2 \widetilde r^E_{i,j}((\phi \cos \theta, \phi \sin \theta))^{q_{i,j}} + O\left ( \frac{1}{E} \right ).
\end{split}
\end{equation}
\end{proposition}
\begin{proof}
By the co-area formula we can rewrite the l.h.s. of \paref{toprove} as 
\begin{equation*}
\begin{split}
&E\int_{\mathcal D} \int_{\mathcal D} \prod_{i,j=0}^2 \widetilde r^E_{i,j}(x-y)^{q_{i,j}}\,dx dy\cr
& = \int_0^{\text{diam}(\mathcal D)} d\phi \underbrace{\int_{\mathcal D} dx \int_{\partial B_\phi(x) \cap \mathcal D} dy\, \prod_{i,j=0}^2 \widetilde r^E_{i,j}(x-y)^{q_{i,j}}}_{=: f(\phi)},
\end{split}
\end{equation*}
where $B_\phi(x) = \lbrace y : \|x - y\| \le \phi\rbrace$, while $\partial B_\phi(x)$ denotes its boundary. 
For $\phi \in [0, \text{inrad}(\mathcal D))$, define 
$$\mathcal D_\phi := \lbrace x\in \mathcal D : B_\phi(x) \subseteq \mathcal D\rbrace,
$$
then 
\begin{equation}\label{sway}
\begin{split}
f(\phi):= &\int_{\mathcal D_\phi} dx \int_{\partial B_\phi(x)} dy\, \prod_{i,j=0}^2 \widetilde r^E_{i,j}(x-y)^{q_{i,j}}\cr
& + \int_{\mathcal D \setminus \mathcal D_\phi} dx \int_{\partial B_\phi(x) \cap \mathcal D} dy\, \prod_{i,j=0}^2 \widetilde r^E_{i,j}(x-y)^{q_{i,j}}.
\end{split}
\end{equation}
Using polar coordinates on $\partial B_\phi (x)$ we can rewrite the first term of the r.h.s. of \paref{sway} as
\begin{equation}\label{1term}
\begin{split}
&\int_{\mathcal D_\phi} dx \int_{\partial B_\phi(x)} dy\, \prod_{i,j=0}^2 \widetilde r^E_{i,j}(x-y)^{q_{i,j}}\cr
& = \text{area}(\mathcal D_\phi) \int_{0}^{2\pi}  \prod_{i,j=0}^2 \widetilde r^E_{i,j}((\phi \cos \theta, \phi \sin \theta))^{q_{i,j}}\phi\, d\theta.
\end{split}
\end{equation}
We have 
\begin{equation*}
\text{area}(\mathcal D_\phi) = \text{area}(\mathcal D) - \text{area}(\mathcal D \setminus \mathcal D_\phi).
\end{equation*}
Now since $\mathcal D\subseteq \mathcal D_\phi + 2\phi B_1$, where $B_1=B_1(0)$ denotes the open ball of radius $1$ centered at $0$, 
\begin{equation}\label{steiner}
\begin{split}
\text{area}(\mathcal D\setminus \mathcal D_\phi) &\le \text{area}(\mathcal D_\phi + 2\phi B_1) - \text{area}(\mathcal D_\phi)\cr
&= 4 W_1(\mathcal D_\phi) \phi + 4 W_2(\mathcal D_\phi) \phi^2,
\end{split}
\end{equation}
where for the last equality we used Steiner formula (for a convex body $K\subseteq \R^2$ and $j=0,1,2$, $W_j(K)$ is the $j$th quermassintegrals) and the equality $W_0(\mathcal D_\psi) = \text{meas}(\mathcal D_\psi)$. Bearing in mind that  if $K\subseteq K'$, then $W_j(K)\le W_j(K')$ for $j=0,1,2$ we find 
\begin{equation}\label{bah}
\text{area}(\mathcal D\setminus \mathcal D_\phi) \le  4 W_1(\mathcal D) \phi + 4 W_2(\mathcal D) \phi^2. 
\end{equation}
Hence we find that  
$$
\text{area}(\mathcal D\setminus \mathcal D_\phi) \int_{0}^{2\pi}  \prod_{i,j=0}^2 \widetilde r^E_{i,j}((\phi \cos \theta, \phi \sin \theta))^{q_{i,j}}\phi\, d\theta = O\left ( \frac{1}{E}   \right )
$$
uniformly for $\phi\in [0, \text{inrad}(\mathcal D)$ by using \paref{bah} and \paref{Ogrande}.
Therefore from \paref{1term} we can write
\begin{equation*}
\begin{split}
&\int_{\mathcal D_\phi} dx \int_{\partial B_\phi(x)} dy\, \prod_{i,j=0}^2 \widetilde r^E_{i,j}(x-y)^{q_{i,j}}\cr
& = \text{area}(\mathcal D) \int_{0}^{2\pi}  \prod_{i,j=0}^2 \widetilde r^E_{i,j}((\phi \cos \theta, \phi \sin \theta))^{q_{i,j}}\phi\, d\theta + O\left ( \frac{1}{E}  \right ).
\end{split}
\end{equation*}
The error term in \paref{sway} can be dealt with as before obtaining 
\begin{equation*}
\begin{split}
\int_{\mathcal D \setminus \mathcal D_\phi} & dx \int_{\partial B_\phi(x) \cap \mathcal D} dy\, \prod_{i,j=0}^2 |\widetilde r^E_{i,j}(x-y)|^{q_{i,j}} \le \int_{\mathcal D \setminus \mathcal D_\phi} dx \int_{\partial B_\phi(x)} dy\, \prod_{i,j=0}^2 |\widetilde r^E_{i,j}(x-y)|^{q_{i,j}}\cr
& = O\left ( \text{area}(\mathcal D \setminus \mathcal D_\phi) \cdot \phi \cdot \frac{1}{E \phi^2}   \right ) = O \left ( \frac{1}{E}\right ).
\end{split}
\end{equation*}
Let us now consider $\phi\in [0, \text{diam}(\mathcal D))$, then 
\begin{equation*}
\begin{split}
f(\phi) &= f(\phi) 1_{[0, \text{inrad}(\mathcal D))}(\phi) +  f(\phi) 1_{[\text{inrad}(\mathcal D), \text{diam}(\mathcal D)}(\phi)\cr
&= \text{area}(\mathcal D) \int_{0}^{2\pi} d\theta\,  \prod_{i,j=0}^2 \widetilde r^E_{i,j}((\phi \cos \theta, \phi \sin \theta))^{q_{i,j}}\phi \cdot 1_{[0, \text{inrad}(\mathcal D))}(\phi) + O\left ( \frac{1}{E}  \right )\cr
 &+ f(\phi) 1_{[\text{inrad}(\mathcal D), \text{diam}(\mathcal D)}(\phi)\cr
 &= \text{area}(\mathcal D) \int_{0}^{2\pi} d\theta\,  \prod_{i,j=0}^2 \widetilde r^E_{i,j}((\phi \cos \theta, \phi \sin \theta))^{q_{i,j}}\phi \cdot 1_{[0, \text{diam}(\mathcal D))}(\phi) \cr
 &+ O\left ( \frac{1}{E}   \right )\cr
 &+ \left( f(\phi) - \text{area}(\mathcal D) \int_{0}^{2\pi} d\theta\,  \prod_{i,j=0}^2 \widetilde r^E_{i,j}((\phi \cos \theta, \phi \sin \theta))^{q_{i,j}}\phi \right ) 1_{[\text{inrad}(\mathcal D), \text{diam}(\mathcal D)}(\phi).
\end{split}
\end{equation*}
Now it suffices to note that 
\begin{equation*}
\begin{split}
&\left | f(\phi) - \text{area}(\mathcal D) \int_{0}^{2\pi} d\theta\,  \prod_{i,j=0}^2 \widetilde r^E_{i,j}((\phi \cos \theta, \phi \sin \theta))^{q_{i,j}}\phi \right |1_{[\text{inrad}(\mathcal D), \text{diam}(\mathcal D)}(\phi) \cr
&\le 2\, \text{area}(\mathcal D) \int_{0}^{2\pi} d\theta\,  \prod_{i,j=0}^2 |\widetilde r^E_{i,j}((\phi \cos \theta, \phi \sin \theta))|^{q_{i,j}}\phi \cdot 1_{[\text{inrad}(\mathcal D), \text{diam}(\mathcal D)}(\phi)\cr
&\le 2\, \text{area}(\mathcal D) \int_{0}^{2\pi} d\theta\,  \prod_{i,j=0}^2 |\widetilde r^E_{i,j}((\phi \cos \theta, \phi \sin \theta))|^{q_{i,j}}\frac{\phi^2}{\text{inrad}(\mathcal D)}. 
\end{split}
\end{equation*}
\end{proof}

In order to study the asymptotic behavior, as $E\to +\infty$, of \paref{toprove}, we need the following 
uniform estimate for Bessel functions \cite[(7)]{Kra}: for $\alpha\ge - 1/2$
\begin{equation}\label{formulabessel}
\frac{1}{\sqrt{2\pi}}\mu \le \sup_{x\ge 0} x^{3/2} \left | J_\alpha(x) - \sqrt{\frac{2}{\pi x}}\cos(x-\omega_\alpha)\right | < \frac{4}{5} \mu, 
\end{equation}
where $\mu:=|\alpha^2 - 1/4|$ and $\omega_\alpha := (2\alpha +1)\pi/4$.
From \paref{formulabessel} we find 
\begin{eqnarray}\label{asymp}
\nonumber r^E((\phi \cos\theta, \phi\sin\theta))  &= & \underbrace{\frac{1}{\pi\sqrt{\sqrt E \phi}}\cos(2\pi \sqrt E\phi-\frac{\pi}{4})}_{=: h^E(\theta)g^E(\phi)} + O\left (\frac{1}{E^{3/4}\phi \sqrt \phi }    \right )\\
\nonumber \widetilde r^E_{0,1}((\phi \cos\theta, \phi\sin\theta)) &=& \underbrace{\frac{\sqrt{2}\cos\theta}{\pi\sqrt{\sqrt E\phi}}\sin (2\pi \sqrt E\phi-\frac{\pi}{4} )}_{=: h^E_{0,1}(\theta)g^E_{0,1}(\phi)} + O\left (\frac{1}{E^{3/4}\phi \sqrt \phi }    \right )\\
\widetilde r^E_{0,2}((\phi \cos\theta, \phi\sin\theta)) &=& \underbrace{\frac{\sqrt{2}\sin \theta}{\pi\sqrt{\sqrt E \phi}}\sin(2\pi \sqrt E\phi-\frac{\pi}{4})}_{=: h^E_{0,2}(\theta) g^E_{0,2}(\phi)} + O\left (\frac{1}{E^{3/4}\phi \sqrt \phi }    \right )\\
\nonumber \widetilde r^E_{1,1}((\phi \cos\theta, \phi\sin\theta)) &=& \underbrace{\frac{2\cos^2\theta}{\pi\sqrt{\sqrt E\phi}}\cos(2\pi \sqrt E\phi-\frac{\pi}{4})}_{=: h^E_{1,1}(\theta)g^E_{1,1}(\phi)}+ O\left (\frac{1}{E^{3/4}\phi \sqrt \phi }    \right )\\
\nonumber \widetilde r^E_{2,2}((\phi \cos\theta, \phi\sin\theta)) &=& \underbrace{\frac{2\sin^2\theta}{\pi\sqrt{\sqrt E\phi}}\cos(2\pi \sqrt E\phi-\frac{\pi}{4})}_{=:h^E_{2,2}(\theta) g^E_{2,2}(\phi)}+ O\left (\frac{1}{E^{3/4}\phi \sqrt \phi }    \right )\\
\nonumber \widetilde r^E_{1,2}((\phi \cos\theta, \phi\sin\theta)) &=&\underbrace{\frac{2\cos\theta\sin\theta}{\pi\sqrt{\sqrt E\phi}}\cos(2\pi \sqrt E\phi-\frac{\pi}{4})}_{=: h^E_{1,2}(\theta)g^E_{1,2}(\phi)} + O\left (\frac{1}{E^{3/4}\phi \sqrt \phi }    \right ),
\end{eqnarray}
uniformly on $(\phi, \theta)$, where the constant involved in the $`O'$-notation does not depend on $E$. 
\begin{proposition}\label{prop_imp}
Let $q_{i,j}\ge 0$ for $i,j=0,1,2$ and $\sum_{i,j=0}^2 q_{i,j} = 4$. Then, as $E\to +\infty$, 
\begin{equation}\label{imp}
\begin{split}
&\int_{\mathcal D} \int_{\mathcal D} \prod_{i,j=0}^2 \widetilde r^E_{i,j}(x-y)^{q_{i,j}}\,dx dy\cr
& = \text{area}(\mathcal D) \int_0^{2\pi}   \prod_{i,j=0}^2 h^1_{i,j}(\theta)^{q_{i,j}}d\theta \cdot \frac{1}{E} \int_1^{\sqrt E \cdot \text{diam}(\mathcal D)}  \psi\prod_{i,j=0}^2 g^1_{i,j}(\psi)^{q_{i,j}}\, d\psi   + O\left ( \frac{1}{E} \right ).
\end{split}
\end{equation}
\end{proposition}
\begin{proof}
Performing a change of variable for the first term in the r.h.s. of \paref{toprove}, we have 
\begin{equation}\label{eqe}
\begin{split}
\text{area}&(\mathcal D) \int_0^{\text{diam}(\mathcal D)}\phi\,  d\phi   \int_0^{2\pi} d\theta  \prod_{i,j=0}^2 \widetilde r^E_{i,j}((\phi \cos \theta, \phi \sin \theta))^{q_{i,j}}\cr
=& \frac{\text{area}(\mathcal D)}{E} \int_0^{\sqrt E \text{diam}(\mathcal D)}\psi\,  d\psi   \int_0^{2\pi} d\theta  \prod_{i,j=0}^2 \widetilde r^1_{i,j}((\psi \cos \theta, \psi \sin \theta))^{q_{i,j}}\cr
=& \frac{\text{area}(\mathcal D)}{E} \int_0^{1}\psi\,  d\psi   \int_0^{2\pi} d\theta  \prod_{i,j=0}^2 \widetilde r^1_{i,j}((\psi \cos \theta, \psi \sin \theta))^{q_{i,j}}\cr
&+ \frac{\text{area}(\mathcal D)}{E} \int_1^{\sqrt E \text{diam}(\mathcal D)}\psi\,  d\psi   \int_0^{2\pi} d\theta  \prod_{i,j=0}^2 \widetilde r^1_{i,j}((\psi \cos \theta, \psi \sin \theta))^{q_{i,j}}
\end{split}
\end{equation}
Since $r^1(\psi \cos \theta, \psi \sin \theta)\to 1$, $\widetilde r^1_{0,i}(\psi \cos \theta, \psi \sin \theta)= O(\psi)$ and $\widetilde r^1_{i,i}(\psi \cos \theta, \psi \sin \theta)\to 1$,  $\widetilde r^1_{1,2}(\psi \cos \theta, \psi \sin \theta)=O(\psi^2)$ as $\psi \to 0$ uniformly on $\theta$ ($i=1,2$), then from \paref{eqe} we have 
\begin{equation}\label{eqf}
\begin{split}
&\frac{\text{area}(\mathcal D)}{E} \int_0^{1}\psi\,  d\psi   \int_0^{2\pi} d\theta  \prod_{i,j=0}^2 \widetilde r^1_{i,j}((\psi \cos \theta, \psi \sin \theta))^{q_{i,j}}\cr
&+ \frac{\text{area}(\mathcal D)}{E} \int_1^{\sqrt E \text{diam}(\mathcal D)}\psi\,  d\psi   \int_0^{2\pi} d\theta  \prod_{i,j=0}^2 \widetilde r^1_{i,j}((\psi \cos \theta, \psi \sin \theta))^{q_{i,j}}\cr
&= O\left ( \frac{1}{E} \right )+ \frac{\text{area}(\mathcal D)}{E} \int_1^{\sqrt E \text{diam}(\mathcal D)}\psi\,  d\psi   \int_0^{2\pi} d\theta  \prod_{i,j=0}^2 \widetilde r^1_{i,j}((\psi \cos \theta, \psi \sin \theta))^{q_{i,j}}.
\end{split}
\end{equation}
Substituting \paref{asymp} into the last term in the r.h.s. of \paref{eqf} we get
\begin{equation}\label{eqg}
\begin{split}
&\frac{\text{area}(\mathcal D)}{E} \int_1^{\sqrt E \cdot \text{diam}(\mathcal D)}\psi\,  d\psi   \int_0^{2\pi} d\theta  \prod_{i,j=0}^2 \widetilde r^1_{i,j}((\psi \cos \theta, \psi \sin \theta))^{q_{i,j}} \cr
&= \text{area}(\mathcal D) \int_0^{2\pi}   \prod_{i,j=0}^2 h^1_{i,j}(\theta)^{q_{i,j}}d\theta \cdot \frac{1}{E} \int_1^{\sqrt E \cdot \text{diam}(\mathcal D)}  \psi\prod_{i,j=0}^2 g^1_{i,j}(\psi)^{q_{i,j}}\, d\psi \cr
&+ O\left ( \frac{1}{E} \int_1^{\sqrt E \cdot \text{diam}(\mathcal D)}  \frac{1}{\psi^2} \right )\cr
&=  \text{area}(\mathcal D) \int_0^{2\pi}   \prod_{i,j=0}^2 h^1_{i,j}(\theta)^{q_{i,j}}d\theta \cdot \frac{1}{E} \int_1^{\sqrt E \cdot \text{diam}(\mathcal D)}  \psi\prod_{i,j=0}^2 g^1_{i,j}(\psi)^{q_{i,j}}\, d\psi + O\left ( \frac{1}{E} \right ).
\end{split}
\end{equation}
Substituting \paref{eqg} into \paref{eqf} we prove \paref{imp}.

\end{proof}

\section{Fourth chaotic components}\label{sec_4}

\subsection{Case of $\mathcal{L}_E$}\label{sec-le}
From Equation \paref{expL} $\mathcal{L}_E[4]$, i.e.,  the projection  of $\mathcal{L}_E$ onto the fourth  chaos, is 
\begin{equation}\label{4chaosIvan}
\begin{split}
\mathcal{L}_E[4]& = \sqrt{2\pi^2E}
\left\{ \beta_4\alpha_{0,0}\int_{\mathcal{D}} H_4(B_E(x)) dx
+\beta_0\alpha_{4,0}\int_{\mathcal{D}} \big(H_4(\widetilde{\partial_1} B_E(x)) +H_4(\widetilde{\partial_2} B_E(x))\big)dx \right.\\
&\left.\hskip1.4cm
+\beta_0\alpha_{2,2}\int_{\mathcal{D}} H_2(\widetilde{\partial_1} B_E(x)) H_2(\widetilde{\partial_2} B_E(x))dx
\right.\\
&\left.\hskip1.4cm
+\beta_2\alpha_{2,0}\int_{\mathcal{D}} H_2(B_E(x))\big(H_2(\widetilde{\partial_1} B_E(x)) +H_2(\widetilde{\partial_2} B_E(x))\big)dx\right\}\\
& = \frac{\sqrt{2\pi^2E}}{128}
\left\{ 8\int_{\mathcal{D}} H_4(B_E(x)) dx
-\int_{\mathcal{D}} \big(H_4(\widetilde{\partial_1} B_E(x)) +H_4(\widetilde{\partial_2} B_E(x))\big)dx \right.\\
&\left.\hskip1.4cm
-2\int_{\mathcal{D}} H_2(\widetilde{\partial_1} B_E(x)) H_2(\widetilde{\partial_2} B_E(x))dx
\right.\\
&\left.\hskip1.4cm
-8\int_{\mathcal{D}} H_2(B_E(x))\big(H_2(\widetilde{\partial_1} B_E(x)) +H_2(\widetilde{\partial_2} B_E(x))\big)dx\right\}\\
&=\frac{\sqrt{2\pi^2E}}{128}
\Big\{8a_{1,E}-a_{2,E}-a_{3,E}-2a_{4,E}-8a_{5,E}-8a_{6,E}\Big\},
\end{split}
\end{equation}
where we used \paref{fewbeta} and \paref{fewalpha}, and we have set
\begin{eqnarray*}
a_{1,E}&:=&\int_{\mathcal{D}} H_4(B_E(x))dx,\quad a_{2,E}:=\int_{\mathcal{D}}  H_4(\widetilde{\partial}_1 B_E(x))dx,\quad
a_{3,E}:=\int_{\mathcal{D}} H_4(\widetilde{\partial}_2 B_E(x))dx,\\
a_{4,E} &:=& \int_{\mathcal{D}}  H_2(\widetilde{\partial}_1 B_E(x))
H_2(\widetilde{\partial}_2 B_E(x))dx,\\
a_{5,E} &:=& \int_{\mathcal{D}}  H_2(B_E(x))
  H_2(\widetilde{\partial}_1 B_E(x))dx,\quad
a_{6,E} := \int_{\mathcal{D}}  H_2(B_E(x))H_2(\widetilde{\partial}_2 B_E(x))dx.
\end{eqnarray*}
\begin{proposition}\label{propvar4L} The variance of the fourth chaotic component \paref{4chaosIvan} of the nodal length satisfies 
\begin{equation}\label{var4L}
\begin{split}
{\rm Var}(\mathcal{L}_E[4]) &= \frac{\pi^2E}{8192}\,{\var}\left (8a_{1,E}-a_{2,E}-a_{3,E}-2a_{4,E}-8a_{5,E}-8a_{6,E}
\right )\\
&\sim \frac{{\rm area}(\mathcal{D})\,\log E}{512\pi},
\end{split}
\end{equation}
where the last asymptotic equivalence holds as $E\to +\infty$. 
\end{proposition}
In order to prove Proposition \ref{propvar4L} we need to find the asymptotics, as $E\to +\infty$, of $\Cov(a_{i,E}, a_{j,E})$ for any $i,j\in \lbrace 1,2,3,4,5,6\rbrace$ (these results are collected in Lemma \ref{oddio1} in Appendix B for simplifying the discussion, and give immediately Proposition \ref{propvar4L}). 

Recall first that whenever  $U,V,W,Z\sim N(0,1)$ are jointly Gaussian with $\E[UV]=\E[WZ]=0$:
\begin{eqnarray}\label{easy}
\nonumber \E[H_2(U)H_2(V)H_2(W)H_2(Z)]
\nonumber &=&4\E[UW]^2\E[V_Z]^2+ 4\E[UZ]^2\E[VW]^2\\
&&+16\E[UW]\E[UZ]\E[VW]\E[VZ]\\
\nonumber \E[H_2(U)H_2(V)H_4(W)]&=&24\E[UW]^2\E[VW]^2\\
\nonumber \E[UVWZ]&=&\E[UW]\E[VZ]+\E[UZ]\E[VW].
\end{eqnarray}
Thanks to \paref{easy}, for any $i,j\in \lbrace 1,2,3,4,5,6\rbrace$, $\Cov(a_{i,E}, a_{j,E})$ can be written as a finite linear combination of 
terms of the same form as the l.h.s. of \paref{imp}. 

Recall now that 
\begin{eqnarray}
\nonumber \cos^2x &=&\frac12 + \frac12 \cos(2x),\label{cos2}\\
\cos^4x &=&\frac38 + \frac18 \cos(4x) + \frac12 \cos(2x),\label{cos4}\\
\nonumber \cos^6x &=&\frac5{16} + \frac{1}{32}\cos(6x) + \frac{3}{16}\cos(4x) +\frac{15}{32}\cos(2x),\label{cos6}\\
\nonumber \cos^8x &=&\frac{35}{128} + \frac{1}{128}(56 \cos(2 x) + 28 \cos(4 x) + 8 \cos(6 x) + \cos(8 x)).\label{cos8}
\end{eqnarray}

Taking advantage of \paref{cos4}, we can find the asymptotic behavior, as $E\to +\infty$, of the first term in the r.h.s. of \paref{imp}, thus obtaining the asymptotic behavior of $\Cov(a_{i,E}, a_{j,E})$ for any $i,j\in \lbrace 1,2,3,4,5,6\rbrace$. 

\subsection{Case of $\mathcal{N}_E$}

Using the results of Section \ref{sec-chaos} we see that  $\mathcal{N}_E[4]$, the projection  of $\mathcal{N}_E$ onto the fourth chaos, is given by
\begin{equation}\label{4chaosPoints}
\mathcal{N}_E[4] =a_E+\widehat{a}_E + b_E,
\end{equation}
where
$$
a_E = \frac{\pi E}{64}\left\{ 8a_{1,E} - a_{2,E} - 2a_{3,E} - 8a_{4,E}\right\}
$$
with $a_{i,E}$, $i=1,\ldots,4$ defined in Section \ref{sec-le},
where $\widehat{a}_E$ is defined the same way than $a_E$ except that we use $\widehat{B}_E$ instead of $B_E$, and where
$$
b_E = \frac{\pi E}{8}\left\{ 2b_{1,E} - b_{2,E} - b_{3,E}
-b_{4,E} - b_{5,E} - \frac14 b_{6,E} -\frac14 b_{7,E} +
\frac54 b_{8,E} + \frac54 b_{9,E}- 3b_{10,E} 
\right\},
$$
with
\begin{eqnarray*}
b_{1,E} &=& \int_{\mathcal{D}}  
H_2(B_E(x))H_2(\widehat{B}_E(x))
dx\\
b_{2,E}&=&\int_{\mathcal{D}} H_2(B_E(x)) H_2(\widetilde{\partial}_1 \widehat{B}_E(x)dx\\
b_{3,E}&=&\int_{\mathcal{D}} H_2(B_E(x)) H_2(\widetilde{\partial}_2 \widehat{B}_E(x))dx\\
b_{4,E}&=&\int_{\mathcal{D}} H_2(\widetilde{\partial}_1 B_E(x))H_2(\widehat{B}_E(x)) dx\\
b_{5,E}&=&\int_{\mathcal{D}} H_2(\widetilde{\partial}_2 B_E(x))H_2(\widehat{B}_E(x)) dx\\
b_{6,E} &=& \int_{\mathcal{D}}  
H_2(\widetilde{\partial}_1 B_E(x))H_2(\widetilde{\partial}_1 \widehat{B}_E(x))
dx\\
b_{7,E} &=& \int_{\mathcal{D}}  
H_2(\widetilde{\partial}_2 B_E(x))H_2(\widetilde{\partial}_2 \widehat{B}_E(x))
dx\\
b_{8,E} &=& \int_{\mathcal{D}}  
H_2(\widetilde{\partial}_1 B_E(x))H_2(\widetilde{\partial}_2 \widehat{B}_E(x))dx\\
b_{9,E} &=& \int_{\mathcal{D}}  
H_2(\widetilde{\partial}_2 B_E(x))H_2(\widetilde{\partial}_1 \widehat{B}_E(x))dx\\
b_{10,E} &=& \int_{\mathcal{D}}  \widetilde{\partial}_1 B_E(x)\widetilde{\partial}_2 B_E(x)
\widetilde{\partial}_1 \widehat{B}_E(x)\widetilde{\partial}_2 \widehat{B}_E(x)dx.
\end{eqnarray*}
\begin{proposition}\label{prop4N} The variance of the fourth chaotic component $\mathcal N_E[4]$ of $\mathcal N_E$ is 
\begin{equation}\label{eqbella}
{\rm Var}(\mathcal{N}_E[4])=2{\rm Var}(a_E)+{\rm Var}(b_E)\sim 
\frac{11{\rm area}(\mathcal{D})}{32\pi}\,E\log E,
\end{equation}
where the last asymptotics holds as $E\to +\infty$. 
\end{proposition}
In order to prove Proposition \ref{prop4N}, observe first from \S \ref{sec-le} that $a_E=\sqrt{2E}\mathcal{L}_E[4]$. As a result, as $E\to\infty$, from Proposition \ref{propvar4L}
\begin{equation}\label{var daje}
{\rm Var}(\widehat{a}_E)={\rm Var}(a_E)\sim \frac{{\rm area}(\mathcal{D})\,\log E}{256\pi}.
\end{equation}
So, it remains to consider $b_E$. From Lemma \ref{oddio2}, which is collected in Appendix B to simplify the discussion, we have the following. 
\begin{lemma}\label{lemmab}
\begin{eqnarray*}
{\rm Var}(b_E)&=&\frac{\pi^2E^2}{64}{\rm Var}\Big(
2b_{1,E} - b_{2,E} - b_{3,E}
-b_{4,E} - b_{5,E} - \frac14 b_{6,E} -\frac14 b_{7,E}\cr 
&& +
\frac54 b_{8,E} + \frac54 b_{9,E}- 3b_{10,E} 
\Big) \sim \frac{43{\rm area}(\mathcal{D})}{128\pi}\,E\log E,
\end{eqnarray*}
where the last asymptotics holds as $E\to +\infty$. 
\end{lemma}
\begin{proof}[Proof of Proposition \ref{prop4N}]
From \paref{var daje}, observing that $a_E$, $\widehat{a}_E$ and $b_E$ are indeed uncorrelated from Lemma \ref{oddio2} and Lemma \ref{lemmab}, we obtain \paref{eqbella}.

\end{proof}

\section{Higher order chaotic components}\label{sec_high}

\subsection{Preliminaries}

Let us start with the following result, whose proof is elementary (see Lemma \ref{lemsmooth}) and hence omitted. 
\begin{lemma}\label{lips}
The map 
$$
\R^2\ni x\mapsto r^E\left (x/ \sqrt E \right )
$$
and its derivatives up to the order
two are Lipschitz with a universal Lipschitz constant $c>0$, in particular independent of $E$. 
\end{lemma}
Let us now consider a square $Q$ of side length $d=\text{diam}(\mathcal D)$ which contains $\mathcal D$, and $M:= \lceil \gamma \sqrt{E}\rceil$, where $\gamma$ will be chosen in a while. Let $\lbrace Q_i : 1, \dots, M^2 \rbrace$ be a partition of $Q$ in $M^2$ squares of side length $d/M$. 
Let $0<\eps < 1/1000$ be a fixed small number, and now choose $\gamma \ge \frac{4c}{\eps}$, where $c$ is the Lipschitz constant in Lemma \ref{lips}. 
The following is inspired by \cite{ORW, RW2}. 
\begin{defn}\label{def_sing}
The pair $(Q_i, Q_j)$ is singular if there exists $(x,y)\in Q_i\times Q_j$, as well as $k,l\in \lbrace 0,1,2 \rbrace$, such that 
$$
|\widetilde r^E_{k,l}(x-y)|\ge \eps.
$$
\end{defn}
\begin{lemma}\label{lemma_sing}
If $(Q_i,Q_j)$ is singular, then $\exists k,l\in \lbrace 0,1,2\rbrace$ such that $\forall (x,y)\in Q_i\times Q_j$ we have 
$$
|\widetilde r^E_{k,l}(x-y)|\ge \frac{\eps}{2}.
$$
\end{lemma}
\begin{proof}
Assume that $(x,y)\in Q_i\times Q_j$ is such that $r^E(x-y)\ge \eps$. 
For $(z,w)\in Q_i\times Q_j$ we have 
\begin{equation*}
\begin{split}
|r^E(z-w) - r^E(x-y)| &\le \left | r^E\left( \frac{(z-w)\sqrt E}{\sqrt E}  \right )  - r^E\left( \frac{(x-y)\sqrt E}{\sqrt E}  \right ) \right |\cr
&\le c\cdot \sqrt{E} \left |   (z-x) - (w-y)   \right |\le 2c\cdot \sqrt E\cdot \frac{1}{M}.
\end{split}
\end{equation*}
It hence follows that 
$$
r^E(z-w) \ge r(x-y) - 2c\cdot \sqrt E\cdot \frac{1}{M}\ge \frac{\eps}{2}. 
$$
The proof for $r^E(x-y)\le -\eps$ is similar, as well as that one in the case of singularities w.r.t. derivatives. 

\end{proof}

For each $Q_i$ consider $\mathcal D \cap Q_i$ and, if it is not empty, set $\mathcal D_i := \mathcal D \cap Q_i$. The set $\lbrace \mathcal D_i \rbrace$ is hence a partition of $\mathcal D$. Let $\mathcal D_1, Q_1$ be the sets containing the origin (note that for sufficiently large $E$, $\mathcal D_1$ and $Q_1$ coincide). 
In view of Lemma \ref{lemma_sing} we give the following. 
\begin{defn}
We say that $(\mathcal D_i, \mathcal D_j)$ is singular if $(Q_i, Q_j)$ is singular. 
\end{defn}
The proof of the following result is analogous to the proof of Lemma 6.3 in \cite{DNPR}, and hence omitted. 
\begin{lemma}\label{Ni}
For a fixed cell $D_i$, the number $N_i$ of $j\in \lbrace 1,2,\dots , M^2\rbrace$ such that $(D_i, D_j)$ is singular is 
$$
N_i = O\left( E\cdot \max_{k,l\in \lbrace 0,1,2 \rbrace} \int_{\mathcal D}\int_{\mathcal D} \widetilde r^E_{k,l}(x-y)^6\,dxdy\right),
$$
where the constant involved in the $'O'$-notation depend nor on $E$ neither on $i$. 
\end{lemma}
The following lemma will be proven in Appendix C. 
\begin{lemma}\label{smallolog}
$\forall k,l\in \lbrace 0,1,2\rbrace$, as $E\to +\infty$, 
$$
\int_{\mathcal D}\int_{\mathcal D} \widetilde r^E_{k,l}(x-y)^6\,dxdy = o\left(\frac{\log E}{E}\right ). 
$$
\end{lemma}
\begin{lemma}\label{L2inD1}
Let $\mathcal L_E(\mathcal D_1)$ denote the nodal length of $B_E$ inside $\mathcal D_1$. Then 
$$
\E \left [\mathcal L_E(\mathcal D_1)^2 \right ] = O\left (\frac{1}{E} \right ). 
$$
\end{lemma}
\begin{proof}
It follows from the proof of Lemma \ref{L2-lemma} that 
\begin{equation*}
\begin{split}
&\E \left [\mathcal L_E(\mathcal D_1)^2 \right ] \cr
&= \int_{\mathcal D_1} \int_{\mathcal D_1} \E[\| \nabla B_E(x)\| \| \nabla B_E(y)\| | B_E(x) = B_E(y) =0] p_{(B_E(x), B_E(y))}(0,0)\,dx dy\cr
&\ll \int_{\mathcal D_1} \int_{\mathcal D_1} \frac{E}{\sqrt E \| x-y\|}\, dx dy = O\left ( \frac{1}{E}\right ). 
\end{split}
\end{equation*}
\end{proof}

\subsection{Residual terms} 

For a random variable $F$ in $L^2(\P)$, let us denote by $F | C_{\ge 6}$ the projection of $F$ onto $C_{\ge 6} := \bigoplus_{q=6}^{+\infty}  C_{q}$.  

Let us start investigating the case of the nodal length. 
We can write 
\begin{equation*}
\begin{split}
\var\left(\mathcal L_E | C_{\ge 6}   \right) 
&= \sum_{(D_i, D_j) \text{ sing.}} \Cov\left(\text{proj}(\mathcal L_E(\mathcal D_i)|C_{\ge 6}),  \text{proj}(\mathcal L_E(\mathcal D_j)|C_{\ge 6})   \right) \cr
&+ \sum_{(D_i, D_j) \text{ non-sing.}} \Cov\left(\text{proj}(\mathcal L_E(\mathcal D_i)|C_{\ge 6}),  \text{proj}(\mathcal L_E(\mathcal D_j)|C_{\ge 6})   \right)\cr
&=: X(E) + Y(E).
\end{split}
\end{equation*}
We are going to separately investigate the two terms $X(E)$ and $Y(E)$. 
\begin{lemma}\label{lem_contr_ns}
The contribution of non-singular pairs of cells is, as $E\to +\infty$, 
\begin{equation*}
Y(E)  = o\left(\log E\right ).
\end{equation*}
\end{lemma}
\begin{proof}
Reasoning as in the second part of the proof of Lemma 2 in \cite{PR}, we find
\begin{equation}\label{train}
\begin{split}
|Y(E)| \le 2\pi^2 E \sum_{q\ge 3} & \sum_{i_1 + i_2 + i_3 = q} \sum_{j_1 + j_2 + j_3 = q} |\beta_{2i_1} \alpha_{2i_2,2 i_3}| |\beta_{2j_1} \alpha_{2j_2, 2j_3}| \times  \cr
&  \times 1_{i_1 + i_2 + i_3 = j_1 + j_2 + j_3} |U(i_1, i_2, i_3, j_1, j_2, j_3)|, 
\end{split}
\end{equation}
where $U(i_1, i_2, i_3, j_1, j_2, j_3)$ (for $i_1+i_2+i_3 = q$) is a sum of at most $(2q)!$ terms of the form 
\begin{equation}\label{non-sing_eq}
\sum_{(D_i, D_j) \text{ non-sing.}}\int_{\mathcal D_i} \int_{\mathcal D_j} \prod_{u=1}^{2q} \widetilde r^E_{l_u, k_u}(x-y)\,dx dy,
\end{equation}
where $l_u, k_u \in \lbrace 0,1,2\rbrace$. Since $2q\ge 6$, and we are working on non-singular pairs of cells (see Definition \ref{def_sing}), from \paref{non-sing_eq} we can write 
\begin{equation}\label{non-sing_eq2}
\left | \sum_{(D_i, D_j) \text{ non-sing.}}\int_{\mathcal D_i} \int_{\mathcal D_j} \prod_{u=1}^{2q} \widetilde r^E_{l_u, k_u}(x-y)\,dx dy \right | \le \eps^{2q-6} \int_{\mathcal D} \int_{\mathcal D} \prod_{u=1}^{6} |\widetilde r^E_{l_u, k_u}(x-y)|\,dx dy.
\end{equation}
Substituting \paref{non-sing_eq2} into \paref{train} we get 
\begin{equation}\label{ciaozia}
\begin{split}
|Y(E)| \le 2\pi^2 E \sum_{q\ge 3}&  (2q)!\sum_{i_1 + i_2 + i_3 = q} \sum_{j_1 + j_2 + j_3 = q} |\beta_{2i_1} \alpha_{2i_2,2 i_3}| |\beta_{2j_1} \alpha_{2j_2, 2j_3}| \times\cr
&\times  (\sqrt{\eps})^{i_1+i_2+i_3} (\sqrt{\eps})^{j_1+j_2+j_3}  \frac{\int_{\mathcal D} \int_{\mathcal D} \prod_{u=1}^{6} |\widetilde r^E_{l_u, k_u}(x-y)|\,dx dy }{\eps^6}.
\end{split}
\end{equation}
Now, as $E\to +\infty$, 
$$
\max_{l_u, k_u \in \lbrace 0,1,2\rbrace}\int_{\mathcal D} \int_{\mathcal D} \prod_{u=1}^{6} |\widetilde r^E_{l_u, k_u}(x-y)|\,dx dy = o\left ( \frac{\log E}{E} \right )
$$
(the proof is analogous to that of Lemma \ref{smallolog} and hence omitted), moreover 
$$
\sum_{q\ge 3}  (2q)! \left ( \sum_{i_1 + i_2 + i_3 = q}(\sqrt{\eps})^{i_1+i_2+i_3} |\beta_{2i_1} \alpha_{2i_2,2 i_3}| \right )^2 < +\infty
$$
which together with \paref{ciaozia} allow to conclude the proof.

\end{proof}

\begin{lemma}\label{lem_contr_s}
The contribution of singular pairs of cells is, as $E\to +\infty$, 
\begin{equation*}
\begin{split}
X(E) = o\left(\log E\right ).
\end{split}
\end{equation*}
\end{lemma}
\begin{proof}
Reasoning as in the first part of the proof of Lemma 2 in \cite{PR} 
\begin{equation*}
\begin{split}
X(E) \ll E \cdot N_1 \cdot \E \left [\mathcal L_E(\mathcal D_1)^2 \right ] \ll E \cdot \max_{k,l\in \lbrace 0,1,2\rbrace} \int_{\mathcal D}\int_{\mathcal D} r^E_{k,l}(x-y)^6\,dxdy = o(\log E),
\end{split}
\end{equation*}
where for the last step we used Lemma \ref{Ni}, Lemma \ref{smallolog} and Lemma \ref{L2inD1}. 

\end{proof}

Let us now investigate the case of nodal points. 
\begin{lemma}\label{lem_contrN}
As $E\to +\infty$, 
\begin{equation*}
\var\left(\mathcal N_E | C_{\ge 6}   \right) = o\left(E \log E  \right).
\end{equation*}
\end{lemma}
\begin{proof}
Let us first write
\begin{equation*}
\begin{split}
\var\left(\mathcal N_E | C_{\ge 6}   \right) 
&= \sum_{(D_i, D_j) \text{ sing.}} \Cov\left(\text{proj}(\mathcal N_E(\mathcal D_i)|C_{\ge 6}),  \text{proj}(\mathcal N_E(\mathcal D_j)|C_{\ge 6})   \right) \cr
&+ \sum_{(D_i, D_j) \text{ non-sing.}} \Cov\left(\text{proj}(\mathcal N_E(\mathcal D_i)|C_{\ge 6}),  \text{proj}(\mathcal N_E(\mathcal D_j)|C_{\ge 6})   \right)\cr
&=: X(E) + Y(E).
\end{split}
\end{equation*}
The contribution of the singular part corresponding to the term $X(E)$ can be dealt with exactly as in the proof of Lemma 3.4 in  \cite{DNPR}, using Lemma \ref{Ni}, Lemma \ref{smallolog} and Lemma \ref{lemtaylor}. 

The remaining term $Y(E)$ which corresponds to the non-singular part can be investigated as in the proof of Lemma 3.5 in \cite{DNPR} being ispired also by the proof of  Lemma \ref{lem_contr_ns}. 

\end{proof}

\section{Proofs of the main results}\label{sec_clt}

\subsection{Central Limit Theorems}

In this section we implicitly represent $B_E$ and its first derivatives in terms of a real Gaussian measure (cf. \paref{isonormal}), allowed by isometric property between Hilbert spaces. 
We prove asymptotic Gaussianity, as $E\to +\infty$, for fourth order components $\mathcal L_E[4]$ and $\mathcal N_E[4]$ in \paref{4chaosIvan} and \paref{4chaosPoints}, respectively. According to \cite{PT05} and because we already checked the convergence of covariances (of summands in both \paref{4chaosIvan} and \paref{4chaosPoints})  in \S \ref{sec_4} (and in Lemmas \ref{oddio1}--\ref{oddio2h}), it suffices to prove that each of those summands satisfies a CLT. To this aim, we apply Fourth Moment Theorem \cite{NP, NuPe}; this technique requires to 
 control the asymptotic behavior of non-trivial contraction norms (see \cite[\S B.4]{NP}) of each term mentioned above. The latter goal is achieved by using the key result contained in the following statement (see the proof of Proposition \ref{prop_clt}).
\begin{lemma}\label{lemma-contr}
Fix integers $1\leq a_1,\ldots,a_4\leq 2$ and $1\leq b_1,\ldots,b_4\leq 3$ such that $b_1+\ldots+b_4=8$. 
Then the quantity
\begin{eqnarray*}
&&\frac{E^2}{\log^2E}\int_{\mathcal{D}^4}\big|J_{a_1}(2\pi\sqrt{E}\|x_1-x_2\|)\big|^{b_1}
\big|J_{a_2}(2\pi\sqrt{E}\|x_2-x_3\|)\big|^{b_2}\\
&&\hskip3cm
\times\big|J_{a_3}(2\pi\sqrt{E}\|x_3-x_4\|)\big|^{b_3}
\big|J_{a_4}(2\pi\sqrt{E}\|x_4-x_1\|)\big|^{b_4}
dx_1\ldots dx_4 = : u_E
\end{eqnarray*}
goes to zero, as $E\to\infty$.
\end{lemma}
\begin{proof}
Performing a change of variables we can write 
\begin{eqnarray*}
u_E&=&\frac{1}{E^2\log^2E}\int_{(\sqrt{E}\mathcal{D})^4}\big|J_{a_1}(2\pi\|x_1-x_2\|)\big|^{b_1}
\big|J_{a_2}(2\pi\|x_2-x_3\|)\big|^{b_2}\\
&&\hskip3cm
\times\big|J_{a_3}(2\pi\|x_3-x_4\|)\big|^{b_3}
\big|J_{a_4}(2\pi\|x_4-x_1\|)\big|^{b_4}
dx_1\ldots dx_4.
\end{eqnarray*}
If, for all $i>j$ we had that $b_i+b_j>4$, then we would have $3(b_1+\ldots+b_4)>24$, which contradicts that $b_1+\ldots+b_4=8$. By symmetry, we can thus assume without loss of generality that $b_1+b_2\leq 4$ and then use that $x^{b_1}y^{b_2}\leq x^{b_1+b_2}+y^{b_1+b_2}$. This way, we get that $u_E$ is less than
\begin{eqnarray}
 \frac{1}{E^2\log^2E} &\int_{(\sqrt{E}\mathcal{D})^4}&
\big|J_{a_1}(2\pi\|x_1-x_2\|)\big|^{b_1+b_2}
\big|J_{a_3}(2\pi\|x_3-x_4\|)\big|^{b_3}\notag\\
&&
\times\big|J_{a_4}(2\pi\|x_4-x_1\|)\big|^{b_4}
dx_1\ldots dx_4
\label{bound}
\end{eqnarray}
plus a similar term.
Now, let us apply the change of variables $u=x_1-x_2$, $v=x_3-x_4$, $w=x_4-x_1$ and $z=x_1$ in (\ref{bound}).
We obtain that (\ref{bound}) is less or equal than
\begin{eqnarray}
&&\frac{{\rm Area}(\mathcal{D})}{E\log^2E}\int_{\sqrt{E}(\mathcal{D}-\mathcal{D})}
\big|J_{a_1}(2\pi\|u\|)\big|^{b_1+b_2}du
\int_{\sqrt{E}(\mathcal{D}-\mathcal{D})}
\big|J_{a_3}(2\pi\|v\|)\big|^{b_3}dv\notag\\
&&\hskip6.5cm\times
\int_{\sqrt{E}(\mathcal{D}-\mathcal{D})}
\big|J_{a_4}(2\pi\|w\|)\big|^{b_4}dw.\label{bound2}
\end{eqnarray}
But $|J_a(2\pi r)|\leq {\rm cst}(a)\,r^{-\frac12}$ for any $r>0$ and $a\in\{0,1,2\}$ so that, for any $b\in\{1,2,3,4\}$,
\begin{equation}\label{control}
\int_{\sqrt{E}(\mathcal{D}-\mathcal{D})}|J_a(2\pi\|u\|)|^bdu\leq {\rm cst}(a,b)\int^{\sqrt{E}}r^{1-\frac{b}2}dr
\leq {\rm cst}(a,b)\left\{
\begin{array}{lll}
E^{1-\frac{b}4}&\,\,\mbox{if $b=1,2,3$}\\
\log E&\,\,\mbox{if $b=4$}
\end{array}
\right..
\end{equation}
Substituting \paref{control} in (\ref{bound2}) and recalling that $1\leq b_1+b_2\leq 4$, $1\leq b_3,b_4\leq 3$ and $b_1+\ldots+b_4=8$, we obtain that (\ref{bound2}) is less or equal than
$$
\frac{{\rm area}(\mathcal{D})}{E\log^2E}\times
E^{1-\frac{b_1+b_2}{4}}\log E
\times E^{1-\frac{b_3}{4}}
\times E^{1-\frac{b_4}{4}} = O((\log E)^{-1})\to 0,\quad\mbox{as $E\to\infty$}.
$$
\end{proof}

We can now prove the main result of this subsection. 
\begin{proposition}\label{prop_clt}
As $E\to +\infty$, 
$$
\frac{
\mathcal{L}_E[4]
}{\sqrt{{\var}(\mathcal{L}_E[4])}} \mathop{\to}^d Z,
$$
and 
$$\frac{
\mathcal{N}_E[4]
}{\sqrt{{\var}(\mathcal{N}_E[4])}} \mathop{\to}^d Z,$$
where $Z$ is a standard Gaussian random variable. 
\end{proposition}
\begin{proof}
Recall the expressions for fourth order chaotic components in \paref{4chaosIvan} and \paref{4chaosPoints}. 
According to \cite{PT05} and because we already checked the convergence of covariances in \S \ref{sec_4}, it remains to check that the non-trivial contractions (see \cite[\S B.4]{NP})
associated with the fourth order Wiener-It\^o integrals
$\sqrt{\frac{E}{\log E}}\,a_{i,E}$ ($1\leq i\leq 6$)
and 
$\sqrt{\frac{E}{\log E}}\,b_{j,E}$ ($1\leq j\leq 10$) in \paref{4chaosIvan} and \paref{4chaosPoints}
go to zero as $E\to\infty$.

Due to the high number of terms that are involved, we only show how to check this on a particular term that is representative of the difficulty. All the other calculations follow exactly the same line, relying on Lemma \ref{lemma-contr}. 

Let us consider 
$$
\sqrt{\frac{E}{\log E}}\,b_{2,E} = I_4\left(\widetilde{\alpha_E}\right),
$$
with
$$
\alpha_E(u_1,\ldots,u_4):=\sqrt{\frac{E}{\log E}}
\int_{\mathcal{D}}f_E(x,u_1)f_E(x,u_2)g_E(x,u_3)g_E(x,u_4)dx.
$$
Here $f_E(x,\cdot)$ and $g_E(x,\cdot)$ are chosen so that
$B_E(x)=I_1(f_E(x,\cdot))$ and $\widetilde{\partial}_1\widehat{B}_E(x)=I_1(g_E(x,\cdot))$ respectively, where $I_k$ indicates a multiple integral of order $k$ with respect to an appropriate real-valued Gaussian measure -- see Remark \ref{r:noisonormal}.
The symmetrization $\widetilde{\alpha_E}$ of $\alpha_E$ is given by
\begin{eqnarray*}
\widetilde{\alpha_E}(u_1,\ldots,u_4)&:=&\frac16\sqrt{\frac{E}{\log E}}
\int_{\mathcal{D}}
\Big\{
f_E(x,u_1)f_E(x,u_2)g_E(x,u_3)g_E(x,u_4)\\
&&
\hskip2.7cm+f_E(x,u_1)g_E(x,u_2)f_E(x,u_3)g_E(x,u_4)\\
&&
\hskip2.7cm+f_E(x,u_1)g_E(x,u_2)g_E(x,u_3)f_E(x,u_4)\\
&&
\hskip2.7cm+g_E(x,u_1)f_E(x,u_2)f_E(x,u_3)g_E(x,u_4)\\
&&
\hskip2.7cm+g_E(x,u_1)f_E(x,u_2)g_E(x,u_3)f_E(x,u_4)\\
&&
\hskip2.7cm+g_E(x,u_1)g_E(x,u_2)f_E(x,u_3)f_E(x,u_4)
\Big\}dx.
\end{eqnarray*}
Let us now consider, for instance, the first contraction
$\widetilde{\alpha_E}\otimes_1 \widetilde{\alpha_E}$.
It is given by a sum of 36 terms. They are all of the same order. For instance,
it contains the term
\begin{eqnarray}
(u_1,u_2,u_3,v_1,v_2,v_3)&\mapsto&\frac{E}{36\log E}\int_{\mathcal{D}^2}
f_E(x_1,u_1)f_E(x_1,u_2)g_E(x_1,u_3)f_E(x_2,v_1)\label{contr}\\
&&\hskip3cm \times g_E(x_2,v_2)g_E(x_2,v_3)\E[\widetilde{\partial}_1\widehat{B}_E(x_1)B_E(x_2)]dx_1dx_2.
\notag
\end{eqnarray}
Then, $\|\widetilde{\alpha_E}\otimes_1 \widetilde{\alpha_E}\|^2$ is given by a sum
of $36^2$ terms, which all behave the same way. One of them (corresponding to (\ref{contr}) above) is given by
\begin{eqnarray}
&&\frac{E^2}{36^2\log^2 E}\int_{\mathcal{D}^4}
\E[B_E(x_1)B_E(x_3)]^2
\E[\widetilde{\partial}_1\widehat{B}_E(x_1)\widetilde{\partial}_1\widehat{B}_E(x_3)]
\E[B_E(x_2)B_E(x_4)]\label{contr2}\\
&&\hskip2.5cm\times
\E[\widetilde{\partial}_1\widehat{B}_E(x_2)\widetilde{\partial}_1\widehat{B}_E(x_4)]^2
\E[\widetilde{\partial}_1\widehat{B}_E(x_1)B_E(x_2)]
\E[\widetilde{\partial}_1\widehat{B}_E(x_3)B_E(x_4)]
dx_1\ldots dx_4\notag.
\end{eqnarray}
Using Lemma \ref{lemsmooth}, we obtain that the absolute value of
(\ref{contr2}) is less or equal than (up to universal constants whose exact value are immaterial)
 \begin{eqnarray*}
&&\frac{E^2}{\log^2 E}\int_{\mathcal{D}^4}
\left(
\big|J_0(2\pi\sqrt{E}\|x_1-x_3\|)\big|^3
+
\big|J_2(2\pi\sqrt{E}\|x_1-x_3\|)\big|^3
\right)\\
&&\hskip2.5cm\times
\left(
\big|J_0(2\pi\sqrt{E}\|x_2-x_4\|)\big|^3
+
\big|J_2(2\pi\sqrt{E}\|x_2-x_4\|)\big|^3
\right)\\
&&\hskip2.5cm\times
\big|J_1(2\pi\sqrt{E}\|x_1-x_2\|)\big|
\times
\big|J_1(2\pi\sqrt{E}\|x_3-x_4\|)\big|
dx_1\ldots dx_4\notag.
\end{eqnarray*}
and thus goes to zero as $E\to\infty$ thanks to Lemma \ref{lemma-contr}.

\end{proof}

\subsection{Proofs of Theorem \ref{thlength} and Theorem \ref{thpoints}}\label{sec_proof}

In this subsection we prove our main results. 

\begin{proof}[Proof of Theorem \ref{thlength}]
Consider the chaotic expansion for the nodal length $\mathcal L_E$ in \paref{expL}. 
For the $0$-th chaotic component we have 
$$
\mathcal L_E[0] = \E[\mathcal L_E] =  \text{area}(\mathcal D)\sqrt{2\pi^2E} \beta_0 \alpha_{0,0} = \text{area}(\mathcal D) \frac{\pi}{\sqrt 2} \sqrt E,
$$
where we used \paref{fewbeta} and \paref{fewalpha}. By \paref{var2L}, \paref{var4L} and Lemma \ref{lem_contr_s}, Lemma \ref{lem_contr_ns} we deduce that, as $E\to +\infty$, 
$$
\var(\mathcal L_E) \sim \var(\mathcal L_E[4])
$$
and 
$$
\frac{\mathcal L_E - \E[\mathcal L_E]}{\sqrt{\var(\mathcal L_E}} = \frac{\mathcal L_E[4]}{\sqrt{\var(\mathcal L_E[4]}} + o_\P(1),
$$
where $o_\P(1)$ denotes a sequence converging to zero in probability. Proposition \ref{prop_clt} allows to conclude the proof.

\end{proof}

\begin{proof}[Proof of Theorem \ref{thpoints}] The proof of this theorem is analogous to the proof of Theorem \ref{thlength}. 
Consider the chaotic expansion for the nodal length $\mathcal N_E$ in \paref{expN}. 
For the $0$-th chaotic component we have 
$$
\mathcal N_E[0] = \E[\mathcal N_E] =  \text{area}(\mathcal D)\cdot 2\pi^2E \cdot  \beta_0^2 \gamma_{0,0,0,0} = \text{area}(\mathcal D) \pi E,
$$
where we used \paref{fewbeta} and \paref{fewgamma}. By \paref{var2N}, \paref{eqbella} and Lemma \ref{lem_contrN} we deduce that, as $E\to +\infty$, 
$$
\var(\mathcal N_E) \sim \var(\mathcal N_E[4])
$$
and 
$$
\frac{\mathcal N_E - \E[\mathcal N_E]}{\sqrt{\var(\mathcal N_E}} = \frac{\mathcal N_E[4]}{\sqrt{\var(\mathcal N_E[4]}} + o_\P(1),
$$
where $o_\P(1)$ denotes a sequence converging to zero in probability. Proposition \ref{prop_clt} allows to conclude the proof.

\end{proof}

\section*{Appendix A}

\begin{proof}[Proof of Lemma \ref{lemsmooth}]
It is a standard fact that for any $m_1, m_2, n_1, n_2\in \N_{\ge 0}$ 
\begin{equation}\label{scambio}
\begin{split}
\E\left [ \frac{\partial^{m_1+m_2}}{\partial x_1^{m_1} \partial x_2^{m_2}} B_E(x)  \frac{\partial^{n_1+n_2}}{\partial y_1^{n_1} \partial y_2^{n_2}} B_E(y)    \right ] &=  \frac{\partial^{m_1 + m_2+n_1 + n_2}}{\partial x_1^{m_1}\partial y_1^{n_1} \partial x_2^{m_2}\partial y_2^{n_2}}  \E[B_E(x) B_E(y)]\cr
&= \frac{\partial^{m_1 + m_2+n_1 + n_2}}{\partial x_1^{m_1}\partial y_1^{n_1} \partial x_2^{m_2}\partial y_2^{n_2}} r^E(x-y),
\end{split}
\end{equation}
where $r^E$ is defined as in \paref{covE}. 
Let us first prove that for $x\in \R^2$, the covariance matrix of the centered Gaussian vector 
$(B_E(x), \nabla B_E(x))$ is 
\begin{equation}\label{matrice}
\begin{pmatrix}
1 &0\\
0 &2\pi^2 E\,I_2
\end{pmatrix},
\end{equation}
where $I_2$ denotes the $2\times 2$-identity matrix. Recall from \paref{int_rep} that the following integral representation holds:
\begin{equation}\label{int_repJ}
J_0(2\pi\sqrt E\|x\|) = \frac{1}{2\pi} \int_{\mathbb S^1} \e^{i2\pi\sqrt E\langle \theta, x\rangle}\,d\theta,\qquad x\in \R^2, 
\end{equation}
where $d\theta$ stands for the uniform measure on the unit circle. 
By \paref{scambio} and \paref{int_repJ}, \paref{matrice} immediately follows. 
Note now that, from \paref{scambio}, in particular we have
\begin{equation}\label{app1}
\E[B_E(x) \partial_1 B_E(y)] = -i\sqrt{E} \int_{\mathbb S^1} \theta_1 \e^{2\pi\sqrt E i \langle \theta, x-y \rangle}\, d\theta;
\end{equation}
in order to find an explicit expression for \paref{app1}, let us first compute $\int_{\mathbb S^1} \theta_1 \e^{i r\langle \theta, u \rangle}\, d\theta$ for $r\in [0,+\infty)$ and any $u\in \mathbb S^1$. 
Let us denote by $r_\tau$ the rotation of angle $\tau$ (the latter is the angle between $\theta$ and $u$), then we have
\begin{equation*}
\begin{split}
\int_{\mathbb{S}^1}\theta_1\,e^{ir\langle \theta,u\rangle}d\theta
&=
\int_{-\pi}^{\pi} (r_\tau(u))_1 \,e^{ir\cos\tau}d\tau\\
&=
\int_{-\pi}^{\pi} (\cos\tau u_1 - \sin\tau u_2) \,e^{ir\cos\tau}d\tau\\
&=-
\int_{-\pi}^{\pi} (\sin\tau u_1 + \cos\tau u_2) \,e^{-ir\sin\tau}d\tau\\
&=-
\int_{-\pi}^{\pi} (\sin\tau u_1 + \cos\tau u_2)\big(\cos(r\sin\tau)-i\sin(r\sin\tau)\big)d\tau\\
&=-
u_2
\int_{-\pi}^{\pi}  \cos\tau \cos (r\sin\tau)d\tau
+iu_1
\int_{-\pi}^{\pi}  \sin\tau \sin (r\sin\tau)d\tau\\
&=
-\pi u_2\big(J_1(r)+J_{-1}(r)\big)
+i \pi u_1\big(J_1(r)-J_{-1}(r)\big)\\
&=
2i\pi u_1J_1(r),
\end{split}
\end{equation*}
where we used integral representation formulas for $\alpha$-order Bessel functions of the first kind $J_\alpha$ \cite[\S 9.1]{AS}, so that, whenever $x\neq y$,
\begin{eqnarray}\label{d1}
\E [B_E(x){\partial}_1 B_E(y) ] &=&
2\pi\sqrt{E} \,\frac{x_1-y_1}{\|x-y\|}\,  J_1(2\pi\sqrt{E}\|x-y\|).
\end{eqnarray}
Analogously, we get 
\begin{eqnarray}\label{d2}
\E [B_E(x){\partial}_2 B_E(y) ] &=&
2\pi\sqrt{E} \,\frac{x_2-y_2}{\|x-y\|}\,  J_1(2\pi\sqrt{E}\|x-y\|);
\end{eqnarray}
 \paref{d1} and \paref{d2} prove \paref{matrixB}. 
 For  $k,l\in\{1,2\}$ from \paref{scambio} and \paref{int_rep}  
we have
\begin{eqnarray*}
\E [\partial_k B_E(x)\partial_l B_E(y) ] &=&
2\pi E\int_{\mathbb{S}^1}z_kz_l\,e^{2i\pi\langle z,\sqrt{E}(x-y)\rangle}dz.
\end{eqnarray*}
Let us first compute $\int_{\mathbb{S}^1} z_1^2\,e^{ir\langle z,u\rangle}dz
$  for $(r,u)\in[0,\infty)\times \mathbb{S}^1$: we have, again with $r_\tau$ denoting the rotation of angle $\tau$,
\begin{eqnarray*}
&&\int_{\mathbb{S}^1} z_1^2\,e^{ir\langle z,u\rangle}dz=
\int_{-\pi}^{\pi} (r_\tau(u))_1^2 \,e^{ir\cos\tau}d\tau\\
&=&
\int_{-\pi}^{\pi} (\cos\tau u_1 - \sin\tau u_2)^2 \,e^{ir\cos\tau}d\tau
=
\int_{-\pi}^{\pi} (\sin\tau u_1 + \cos\tau u_2)^2 \,e^{-ir\sin\tau}d\tau\\
&=&
\int_{-\pi}^{\pi} \big(\sin^2\tau\, u_1^2 + \cos^2\tau\, u_2^2
+2\cos\tau \sin\tau u_1u_2\big) \big(\cos(r\sin\tau)-i\sin(r\sin\tau)\big)d\tau\\
&=&
\int_{-\pi}^{\pi}  \big(\sin^2\tau\, u_1^2 + \cos^2\tau\, u_2^2\big)\cos(r\sin\tau)d\tau
-i\, u_1u_2\,\int_{-\pi}^{\pi} \sin(2\tau)\,\sin(r\sin\tau)d\tau\\
&=&
\frac{ 1}{2}\int_{-\pi}^{\pi}  \big(1+(1-2u_1^2)\cos(2\tau)\big) \cos(r\sin\tau)d\tau\\
&=&
\pi J_0(r) + \frac{\pi}2(1-2u_1^2)(J_2(r)+J_{-2}(r)) =
\pi J_0(r) + (1-2u_1^2)\pi J_2(r) .
\end{eqnarray*}
Similarly
\begin{eqnarray*}
\int_{\mathbb{S}^1} z_2^2\,e^{ir\langle z,u\rangle}dz=
\pi J_0(r) + (1-2u_2^2)\pi J_2(r) =
\pi J_0(r) + (2u_1^2-1)\pi J_2(r),
\end{eqnarray*}
whereas
\begin{eqnarray*}
&&\int_{\mathbb{S}^1} z_1z_2\,e^{ir\langle z,u\rangle}dz=
\int_{-\pi}^{\pi} (r_\tau(u))_1(r_\tau(u))_2 \,e^{ir\cos\tau}d\tau\\
&=&
\int_{-\pi}^{\pi} (\cos\tau u_1 - \sin\tau u_2)(\sin\tau u_1 + \cos\tau u_2) \,e^{ir\cos\tau}d\tau\\
&=&-
\int_{-\pi}^{\pi} (\sin\tau u_1 + \cos\tau u_2)(\cos\tau u_1 - \sin\tau u_2) \,e^{-ir\sin\tau}d\tau\\
&=&
\int_{-\pi}^{\pi} \big(\frac12\sin(2\tau)\, (1-2u_1^2)
-\cos(2\tau) u_1u_2\big) \big(\cos(r\sin\tau)-i\sin(r\sin\tau)\big)d\tau\\
&=&- u_1u_2\int_{-\pi}^{\pi}  \cos(2\tau)\,\cos(r\sin\tau)d\tau\\
&=&-u_1u_2\pi (J_2(r)+J_{-2}(r)) = -2u_1u_2\pi J_2(r).
\end{eqnarray*}
Thus, when $x\neq y$,
\begin{eqnarray*}
\E [{\partial}_1 B_E(x){\partial}_1 B_E(y) ]
&=& 2\pi^2 E \left ( J_0(2\pi\sqrt{E}\|x-y\|) + \left(1-2\frac{(x_1-y_1)^2}{\|x-y\|^2}\right)J_2(2\pi\sqrt{E}\|x-y\|)\right )\\
\E [{\partial}_2 B_E(x){\partial}_2 B_E(y) ]
&=& 2\pi^2 E \left ( J_0(2\pi\sqrt{E}\|x-y\|) + \left(1-2\frac{(x_2-y_2)^2}{\|x-y\|^2}\right)J_2(2\pi\sqrt{E}\|x-y\|) \right )\\
\E [{\partial}_1 B_E(x){\partial}_2 B_E(y) ] &=&-4\pi^2E \frac{(x_1-y_1)(x_2-y_2)}{\|x-y\|^2}J_2(2\pi\sqrt{E}\|x-y\|),
\end{eqnarray*}
which are  \paref{covii} and \paref{cov12}.  

\end{proof}

The following result concerns some (known) properties of the nodal sets of $B_E$ and its complex version.
\begin{lemma}\label{lemnodal}
\begin{enumerate}
\item The value $0$ is not singular for $B_E$ a.s., i.e. 
$$\P(\exists x : B_E(x)=0, \nabla B_E(x)=0) = 0;$$
\item the nodal set $B_E^{-1}(0)$ is a smooth one dimensional submanifold of $\R^2$ a.s.; 
\item $B_E^{-1}(0) \cap \partial \mathcal D$ consists of a finite number of points a.s.; 
\item the nodal set $(B_E^{\mathbb C})^{-1}(0) = B_E^{-1}(0)\cap \widehat B_E^{-1}(0)$ consists of isolated points a.s.;
\item the number of nodal points $(B_E^{\mathbb C})^{-1}(0) $ in $\mathcal D$ is a.s. finite and none of them lies on $\partial D$ a.s.
\end{enumerate}
\end{lemma}
\begin{proof}
1.  Proposition 6.12 in \cite{AW} ensures that $0$ is not a singular value a.s. Indeed, the hypothesis of Proposition 6.12 are satisfied, the random variables 
$B_E(x)$, $\partial_1 B_E(x)$, $\partial_2 B_E(x)$ being independent for fixed $x\in \R^2$ (Point 2. in Lemma \ref{lemsmooth}). 

2. It follows from Point 1 by Sard's lemma. 

3. Let $\gamma$ be a unit speed parameterization of the boundary $\partial \mathcal D$. 
The restriction of $B_E$ to $\partial \mathcal D$ is the one-dimensional Gaussian process $t\mapsto B_E(\gamma(t))$ whose first time-derivative is 
\begin{equation}\label{time-der}
B_E(\gamma(t))' = \langle \nabla B_E(\gamma(t)), \dot \gamma(t)\rangle. 
\end{equation}
From \paref{time-der} and Point 1.  we deduce that 
$$
\P(\exists t : B_E(\gamma(t))=B_E(\gamma(t))'=0)=0, 
$$
i.e. the value $0$ is not singular a.s. for $B_E(\gamma)$, hence the zeros of $B_E$ on $\partial \mathcal D$ are isolated points a.s. (by a standard application of the inverse mapping theorem), and their number is finite (see \cite[p.269]{AT}).

4. Let us consider the two-dimensional Gaussian field on the plane $(B_E, \widehat B_E)$, where we recall $\widehat B_E$ to be an independent copy of $B_E$. In view of Point 1., the value $(0,0)$ is not singular for $(B_E, \widehat B_E)$, hence a standard application of the inverse mapping theorem entails that the common zeros of $B_E$ and $\widehat B_E$ are isolated points. 

5. The value $0$ being not singular for $(B_E, \widehat B_E)$, from \cite[p.269]{AT} the number of nodal points in $\mathcal D$ is finite a.s. We can apply Lemma 11.2.10 in \cite{AT} to the two-dimensional random field $(B_E, \widehat B_E)$ restricted to the boundary $\partial \mathcal D$ to get that $(B_E^{\mathbb C})^{-1}(0) \cap \partial \mathcal D = \emptyset$ a.s. 

\end{proof}

\begin{proof}[Proof of Lemma \ref{lemma-as-conv}]
We can rewrite \paref{eps-approx} by means of the co-area formula \cite[Proposition 6.13]{AW} as 
\begin{equation}\label{coareaL}
\mathcal L_E^\eps = \frac{1}{2\eps} \int_{-\eps}^{\eps} \text{length}(B_E^{-1}(s)\cap \mathcal D)\,ds,
\end{equation}
where $B_E^{-1}(s)=\lbrace x\in \R^2 : B_E(x) = s\rbrace$. 
Theorem 3 in \cite{APP} ensures that the map $s\mapsto \text{length}(B_E^{-1}(s))$ is a.s. continuous at $0$, so that by the Foundamental Theorem of Calculus we have
$$
\lim_{\eps \to 0} \mathcal L_E^\eps = \lim_{\eps \to 0} \frac{1}{2\eps} \int_{-\eps}^{\eps} \text{length}(B_E^{-1}(s)\cap \mathcal D)\,ds = \mathcal L_E, \qquad a.s. 
$$
In order to prove \paref{as-convN} we apply Theorem 11.2.3 in \cite{AT}, the hypothesis being satisfied.

\end{proof}

\begin{proof}[Proof of Lemma \ref{L2-lemma}]
We have $\mathcal L_E \in L^2(\mathbb P)$, the nodal length of $B_E$ being a.s. bounded in $\mathcal D$ \cite{DF}. 
The collection of random variables $\lbrace \mathcal L^\eps_E \rbrace_{\eps >0}$ is in $L^2(\P)$ since 
$$
\mathcal L_E^\eps \le \frac{1}{2\eps} \int_{\mathcal D} \| \nabla B_E(x)\|\,dx,
$$
hence 
\begin{equation*}
\begin{split}
\E[(\mathcal L_E^\eps)^2] &\le \frac{1}{4\eps^2} \int_{(\mathcal D)^2} \E[\| \nabla B_E(x)\|\cdot \| \nabla B_E(y)\|]\,dx dy\cr
&\le \area(\mathcal D) \frac{1}{4\eps^2} \int_{\mathcal D}\E[\| \nabla B_E(x)\|^2]\,dx = (\area(\mathcal D))^2 \frac{\pi^2 E}{\eps^2} < +\infty. 
\end{split}
\end{equation*}
In view of Lemma \ref{lemma-as-conv}, in order to prove that $\mathcal L_E^\eps$ converges to the nodal length in $L^2(\P)$ it suffices to show that 
\begin{equation}\label{app2}
\lim_{\eps \to 0} \E[(\mathcal L_E^\eps)^2] = \E[\mathcal L_E^2]
\end{equation}
(see also \cite[Lemma 7.2.1]{Ros15}). 
By Fatou's lemma and \paref{coareaL} we get
$$
\E[\mathcal L_E^2] \le \liminf_{\eps\to 0} \E[(\mathcal L_E^\eps)^2] \le \limsup_{\eps \to 0} \E\left[\left(\frac{1}{2\eps}\int_{-\eps}^\eps \mathcal L_E(s)\,ds \right )^2\right].
$$
By Jensen's inequality 
$$
\E[\mathcal L_E^2] \le \limsup_{\eps \to 0} \E\left[\left(\frac{1}{2\eps}\int_{-\eps}^\eps \mathcal L_E(s)\,ds \right )^2\right]\le \limsup_{\eps \to 0} \frac{1}{2\eps}\int_{-\eps}^\eps \E\left[\mathcal L_E(s) ^2\right]\,ds = \E\left[\mathcal L_E ^2\right],
$$
the last step following from the continuity of the map $s\mapsto \E[\mathcal L_E(s)^2]$
at $0$ which will be proven just below. Standard Kac-Rice formula \cite[Theorem 6.9]{AW} allows to write
\begin{equation}
\begin{split}
 &\E[\mathcal L_E(s)^2] \cr
&= \int_{(\mathcal D)^2} \E[\| B_E(x)\| \| B_E(y)\|| B_E(x)=s, B_E(y) = s] p_{(B_E(x), B_E(y))}(s,s)\,dx dy,
\end{split}
\end{equation}
where $p_{(B_E(x), B_E(y))}$ denotes the density of the random vector $(B_E(x), B_E(y))$. 
It suffices to show that there exists a measurable function $g=g(x,y)$ integrable on $(\mathcal D)^2$ such that 
$$
 \E[\| B_E(x)\| \| B_E(y)\| | B_E(x)=s, B_E(y) = s]\, p_{(B_E(x), B_E(y))}(s,s) \le g(x,y),\qquad \forall s.
$$
It is immediate that 
$$
p_{(B_E(x), B_E(y))}(s,s) \le p_{(B_E(x), B_E(y))}(0,0) = \frac{1}{2\pi\sqrt{1-J_0(2\pi\sqrt{E}\|x-y\|)^2}}.
$$
From Lemma \ref{lemsmooth}, the vector $\nabla B_E(x)$ conditioned to $B_E(x)=B_E(y)=s$ is Gaussian with mean 
$$
s \frac{\nabla_x r^E(x-y)}{1+r^E(x-y)}
$$
and covariance matrix 
\begin{eqnarray}\label{Omega}
&& \Omega^E(x-y) \\ 
&&:= 2\pi^2 E I_2 - \frac{1}{1-r^E(x-y)^2} \begin{pmatrix} (\partial_{x_1} r^E(x-y))^2 &\partial_{x_1} r^E(x-y) \partial_{x_2} r^E(x-y)\\ 
\partial_{x_1} r^E(x-y) \partial_{x_2} r^E(x-y) &(\partial_{x_2} r^E(x-y))^2
 \end{pmatrix}. \notag
\end{eqnarray}
Jensen's inequality yields 
\begin{equation}
\begin{split}
\E[ &\| B_E(x)\|   \| B_E(y)\|  | B_E(x)=s, B_E(y) = s]\le  \E[\| B_E(x)\|^2 | B_E(x)=s, B_E(y) = s]\cr
&= \var( \partial_1 B_E(x) | B_E(x)=s, B_E(y) = s) + \var( \partial_2 B_E(x) | B_E(x)=s, B_E(y) = s) \cr
 &+ \E[\partial_1 B_E(x) | B_E(x)=s, B_E(y) = s]^2 +  \E[\partial_2 B_E(x) | B_E(x)=s, B_E(y) = s]^2\cr
&=  4\pi^2E -\frac{4\pi^2E J_1(2\pi\sqrt E \|x-y\|)^2}{1-J_0(2\pi\sqrt{E}\|x-y\|)^2} + s^2\frac{4\pi^2E J_1(2\pi\sqrt{E} \|x-y\|)^2}{(1+ J_0(2\pi\sqrt{E}\|x-y\|))^2} \cr
&\le 2\pi^2 E + s^2\frac{4\pi^2E J_1(2\pi\sqrt{E} \|x-y\|)^2}{(1+ J_0(2\pi\sqrt{E}\|x-y\|))^2}\le 2\pi^2 E + \delta^2\frac{4\pi^2E J_1(2\pi\sqrt{E} \|x-y\|)^2}{(1+ J_0(2\pi\sqrt{E}\|x-y\|))^2},
\end{split}
\end{equation}
for any $\delta>0$. If we set 
$$
g(x,y) := 2\pi^2 E + \delta^2\frac{4\pi^2E J_1(2\pi\sqrt{E} \|x-y\|)^2}{(1+ J_0(2\pi\sqrt{E}\|x-y\|))^2},
$$
then the proof of \paref{2conv} is concluded. 

The proof of \paref{2convN} relies on the same argument as that of \paref{2conv}. Let us first show that $\mathcal N_E\in L^2(\P)$. Theorem 6.3 in \cite{AW} ensures that the second factorial moment of $\mathcal N_E$ has the following integral representation
\begin{equation}
\begin{split}
&\E[\mathcal N_E  (\mathcal N_E-1)] \cr
&= \int_{(\mathcal D)^2} \E\left[|\text{Jac}_{B_E, \widehat B_E}(x)|    |\text{Jac}_{B_E, \widehat B_E}(y)| | B_E(x) = 0, B_E(y) = 0, \widehat B_E(x) = 0, \widehat B_E(y) = 0    \right]\times \cr
&\quad\quad\times p_{(B_E(x), B_E(y))}(0,0)\,dx dy.
\end{split}
\end{equation}
Reasoning as in the proof of \cite[Lemma 3.4]{DNPR}, we have 
\begin{equation}\label{ciaone}
\begin{split}
&\E\left[|\text{Jac}_{B_E, \widehat B_E}(x)|    |\text{Jac}_{B_E, \widehat B_E}(y)| | B_E(x) = 0, B_E(y) = 0, \widehat B_E(x) = 0, \widehat B_E(y) = 0    \right] \cr
&\ll \frac{\text{det}(\Omega^E(x-y))}{1- J_0(2\pi\sqrt E \| x-y\|)^2},
\end{split}
\end{equation}
where, for any $s\in \R$, $\Omega^E(x-y)$ denotes the covariance matrix of $\nabla B_E(x)$ conditioned to $B_E(x) = B_E(y)= s$. 
Lemma \ref{lemtaylor} ensures that the double integral over $\mathcal D$ of the rhs of \paref{ciaone} is finite.

Let us now prove that the map $s\mapsto \E[\mathcal N_E(s)^2]$ is continuous at $0$. 
Note first that we can write
\begin{equation}
\E[\mathcal N_E(s)^2] = \E[\mathcal N_E(s)(\mathcal N_E(s)-1)] + \E[\mathcal N_E(s)].
\end{equation}
To evaluate the mean, we use Kac-Rice formula \cite[Thereom 6.2]{AW} and Lemma \ref{lemsmooth}
\begin{equation}
\E[\mathcal N_E(s)] = \int_{\mathcal D} \E\left [|\text{Jac}_{B_E, \widehat B_E}(x)|   \right ] p_{(B_E(x), \widehat B_E(x))}(s,s)\,dx.
\end{equation}
Since $ \E\left [|\text{Jac}_{B_E, \widehat B_E}(x)|   \right ] =2\pi^2E$ and $p_{(B_E(x), \widehat B_E(x))}(s,s)\le \frac{1}{2\pi}$ for every $s$, then $s\mapsto \E[\mathcal N_E(s)]$ is continuous. 

Let us now deal with the second factorial moment, again using Kac-Rice formula \cite[Theorem 6.3]{AW}.
\begin{equation}
\begin{split}
\E[\mathcal N_E(s) & (\mathcal N_E(s)-1)] \cr
= \int_{(\mathcal D)^2} & \E\Big [|\text{Jac}_{B_E, \widehat B_E}(x)|  |\text{Jac}_{B_E, \widehat B_E}(y)| | B_E(x)=\widehat B_E(x)=B_E(y)= \widehat B_E(y)=s\Big ]\times \cr
&\times  p_{(B_E(x), \widehat B_E(x),B_E(y), \widehat B_E(y))}(s,s,s,s)\,dxdy.
\end{split}
\end{equation}
Jensen's inequality yields 
\begin{equation}
\begin{split}
\E\Big [|\text{Jac}(x)|&  |\text{Jac}(y)| | B_E(x)=\widehat B_E(x)=B_E(y)= \widehat B_E(y)=s\Big ]\cr
&\le \E\Big [|\text{Jac}(x)|^2 | B_E(x)=\widehat B_E(x)=B_E(y)= \widehat B_E(y)=s\Big ]\cr
&= 2\left(\E[X^2] \E[Y^2]  - \E[X Y]^2 \right),
\end{split}
\end{equation}
where $(X,Y)$ is a random vector with the same distribution as $\nabla B_E(x) | B_E(x) = B_E(y) =s$. 
Hence some straightforward computations lead to 
\begin{equation}
\begin{split}
\E[X^2] \E[Y^2]  - \E[X Y]^2 &= 2\pi^2E \left(2\pi^2E - \frac{(\partial_{x_1} r^E(x-y))^2 + (\partial_{x_2} r^E(x-y))^2}{1 - r^E(x-y)^2}    \right)\cr
&+2\pi^2E s^2 \frac{(\partial_{x_1} r^E(x-y))^2 + (\partial_{x_2} r^E(x-y))^2}{(1+r^E(x-y))^2}\cr
& +s^2 \frac{(\partial_{x_1} r^E(x-y))^2\cdot (\partial_{x_2} r^E(x-y))^2}{(1+r^E(x-y))^2(1-r^E(x-y)^2)}\cr
&= 2\pi^2E \left(2\pi^2E - \frac{(\partial_{x_1} r^E(x-y))^2 + (\partial_{x_2} r^E(x-y))^2}{1 - r^E(x-y)^2}    \right)\cr
&+2\pi^2E \delta^2 \frac{(\partial_{x_1} r^E(x-y))^2 + (\partial_{x_2} r^E(x-y))^2}{(1+r^E(x-y))^2}\cr
& +\delta^2 \frac{(\partial_{x_1} r^E(x-y))^2\cdot (\partial_{x_2} r^E(x-y))^2}{(1+r^E(x-y))^2(1-r^E(x-y)^2)}.
\end{split}
\end{equation}
which is integrable on $\mathcal D \times \mathcal D$.

\end{proof}

\section*{Appendix B}

\begin{lemma}\label{oddio1}
As $E\to +\infty$, we have 
\begin{eqnarray*}
(i)\quad  {\rm Var}(a_{1,E})&=&24\int_{\mathcal{D}}\int_{\mathcal{D}} r^E(x-y)^4
dxdy
\sim9\,\frac{{\rm area}(\mathcal{D})}{\pi^3}\times \frac{\log E}{E},\\
{\rm Cov}(a_{1,E},a_{2,E})&=&24\int_{\mathcal{D}}\int_{\mathcal{D}}\widetilde r_{0,1}^E(x-y)^4\,dxdy
\sim\frac{27}{2}\,\frac{{\rm area}(\mathcal{D})}{\pi^3}\times \frac{\log E}{E},\\
{\rm Cov}(a_{1,E},a_{3,E})&=&24\int_{\mathcal{D}}\int_{\mathcal{D}}\widetilde r_{0,2}^E(x-y)^4\,dxdy
\sim\frac{27}{2}\,\frac{{\rm area}(\mathcal{D})}{\pi^3}\times \frac{\log E}{E},\\
{\rm Cov}(a_{1,E},a_{3,E})&=&24\int_{\mathcal{D}}\int_{\mathcal{D}}\widetilde r_{0,1}^E(x-y)^2\widetilde r_{0,2}^E(x-y)^2\,dxdy
\sim\frac92\,\frac{{\rm area}(\mathcal{D})}{\pi^3}\times \frac{\log E}{E},\\
{\rm Cov}(a_{1,E},a_{5,E})&=&24\int_{\mathcal{D}}\int_{\mathcal{D}}r^E(x-y)^2\widetilde r_{0,1}^E(x-y)^2\,dxdy
\sim3\,\frac{{\rm area}(\mathcal{D})}{\pi^3}\times \frac{\log E}{E},\\
{\rm Cov}(a_{1,E},a_{6,E})&=&24\int_{\mathcal{D}}\int_{\mathcal{D}}r^E(x-y)^2\widetilde r_{0,2}^E(x-y)^2\,dxdy
\sim3\,\frac{{\rm area}(\mathcal{D})}{\pi^3}\times \frac{\log E}{E}.
\end{eqnarray*}
\end{lemma}
\begin{proof}
Let us prove (i). From Proposition \ref{prop_imp}, 
\begin{equation}
\begin{split}
\var(a_{1,E}) &= 24 \int_{\mathcal D} \int_{\mathcal D} r^E(x-y)^4\,dx dy \cr
&= 24 \text{area}(\mathcal D) \frac{2\pi}{E} \int_1^{\sqrt E \cdot \text{diam}(\mathcal D)}  \psi \left ( \frac{1}{\pi \sqrt \psi} \cos \left ( 2\pi \psi - \frac{\pi}{4}  \right )  \right )^4\, d\psi   + O\left ( \frac{1}{E} \right ) \cr
&=  24 \text{area}(\mathcal D) \frac{2}{\pi^3 E} \int_1^{\sqrt E \cdot \text{diam}(\mathcal D)} \frac{1}{\psi} \cos^4 \left ( 2\pi \psi - \frac{\pi}{4}  \right )\, d\psi   + O\left ( \frac{1}{E} \right ).
\end{split}
\end{equation}
Thanks to \paref{cos4} we have that, as $E\to +\infty$, 
\begin{equation*}
\begin{split}
24 \text{area}(\mathcal D) \frac{2}{\pi^3 E} \int_1^{\sqrt E \cdot \text{diam}(\mathcal D)} \frac{1}{\psi} \cos^4 \left ( 2\pi \psi - \frac{\pi}{4}  \right )\, d\psi &\sim 24 \text{area}(\mathcal D) \frac{2}{\pi^3 E} \cdot \frac38 \cdot \log \sqrt E\cr
&= \frac{9}{\pi^3 E} \text{area}(\mathcal D) \log E,
\end{split}
\end{equation*}
that allows to conclude. 
The proof for the remaining terms is analogous to the proof of (i), and hence omitted.
 
\end{proof}

The proofs of the following lemmas follow from an application of Proposition \ref{prop_imp}, completely analogous to the one appearing in the proof of Lemma \ref{oddio1}. 
\begin{lemma}\label{oddio1a}
As $E\to +\infty$, we have 
\begin{eqnarray*}
{\rm Var}(a_{2,E})
&=&24\int_{\mathcal{D}}\int_{\mathcal D}\widetilde r_{1,1}^E(x-y)^4 \,dxdy
\sim\frac{315}{8}\,\frac{{\rm area}(\mathcal{D})}{\pi^3}\times \frac{\log E}{E},\\
{\rm Cov}(a_{2,E},a_{3,E})&=&24\int_{\mathcal{D}}\int_{\mathcal D}\widetilde r_{1,2}^E(x-y)^4\,dxdy
\sim\frac{27}8\,\frac{{\rm area}(\mathcal{D})}{\pi^3}\times \frac{\log E}{E},\\
{\rm Cov}(a_{2,E},a_{4,E})&=&24\int_{\mathcal{D}}\int_{\mathcal D}\widetilde r_{1,1}^E(x-y)^2\widetilde r_{1,2}^E(x-y)^2\,dxdy
\sim\frac{45}8\,\frac{{\rm area}(\mathcal{D})}{\pi^3}\times \frac{\log E}{E},\\
{\rm Cov}(a_{2,E},a_{5,E})&=&24\int_{\mathcal{D}}\int_{\mathcal D}\widetilde r_{0,1}^E(x-y)^2\widetilde r_{1,1}^E(x-y)^2\,dxdy
\sim\frac{15}2\,\frac{{\rm area}(\mathcal{D})}{\pi^3}\times \frac{\log E}{E},\\
{\rm Cov}(a_{2,E},a_{6,E})&=&24\int_{\mathcal{D}}\int_{\mathcal D}\widetilde r_{0,1}^E(x-y)^2\widetilde r_{1,2}^E(x-y)^2\,dxdy
\sim\frac{3}2\,\frac{{\rm area}(\mathcal{D})}{\pi^3}\times \frac{\log E}{E}.
\end{eqnarray*}
\end{lemma}
\begin{lemma}\label{oddio1b}
As $E\to +\infty$, we have 
\begin{eqnarray*}
{\rm Var}(a_{3,E})
&=&24\int_{\mathcal{D}}\int_{\mathcal D}\widetilde r_{2,2}^E(x-y)^4\,dxdy
\sim\frac{315}{8}\,\frac{{\rm area}(\mathcal{D})}{\pi^3}\times \frac{\log E}{E},\\
{\rm Cov}(a_{3,E},a_{4,E})&=&24 \int_{\mathcal{D}}\int_{\mathcal D}\widetilde r_{2,2}^E(x-y)^2\widetilde r_{1,2}^E(x-y)^2 \,dxdy
\sim\frac{45}8\,\frac{{\rm area}(\mathcal{D})}{\pi^3}\times \frac{\log E}{E},\\
{\rm Cov}(a_{3,E},a_{5,E})&=&24 \int_{\mathcal{D}}\int_{\mathcal D}\widetilde r_{0,2}^E(x-y)^2\widetilde r_{1,2}^E(x-y)^2\,dxdy
\sim\frac{3}2\,\frac{{\rm area}(\mathcal{D})}{\pi^3}\times \frac{\log E}{E},\\
{\rm Cov}(a_{3,E},a_{6,E})&=&24\int_{\mathcal{D}}\int_{\mathcal D}\widetilde r_{0,2}^E(x-y)^2\widetilde r_{2,2}^E(x-y)^2\,dxdy
\sim\frac{15}2\,\frac{{\rm area}(\mathcal{D})}{\pi^3}\times \frac{\log E}{E}.
\end{eqnarray*}
\end{lemma}
\begin{lemma}\label{oddio1c}
As $E\to +\infty$, we have 
\begin{eqnarray*}
{\rm Var}(a_{4,E})&=&4\int_{\mathcal{D}}\int_{\mathcal D}(\widetilde r_{1,1}^E(x-y)^2\widetilde r_{2,2}^E(x-y)^2+\widetilde r_{1,2}^E(x-y)^4\\
&+&4\widetilde r_{1,1}^E(x-y) \widetilde r_{2,2}^E(x-y)\widetilde r_{1,2}^E(x-y)^2)\,dxdy
\sim \frac{27}{8}\,\frac{{\rm area}(\mathcal{D})}{\pi^3}\times \frac{\log E}{E},\\
{\rm Cov}(a_{4,E},a_{5,E})&=&4\int_{\mathcal{D}}\int_{\mathcal D}(\widetilde r_{0,1}^E(x-y)^2\widetilde r_{1,2}^E(x-y)^2+\widetilde r_{0,2}^E(x-y)^2\widetilde r_{1,1}^E(x-y)^2\\
&+&4\widetilde r_{0,1}^E(x-y)\widetilde r_{0,2}^E(x-y)\widetilde r_{1,1}^E(x-y)\widetilde r_{1,2}^E(x-y))\,dxdy
\sim\frac{3}2\,\frac{{\rm area}(\mathcal{D})}{\pi^3}\times \frac{\log E}{E},\\
{\rm Cov}(a_{4,E},a_{6,E})&=&4\int_{\mathcal{D}}\int_{\mathcal D}(\widetilde r_{0,1}^E(x-y)^2\widetilde r_{2,2}^E(x-y)^2+\widetilde r_{0,2}^E(x-y)^2\widetilde r_{1,2}^E(x-y)^2\\
&+& 4\widetilde r_{0,1}^E(x-y)\widetilde r_{0,2}^E(x-y)\widetilde r_{2,2}^E(x-y)\widetilde r_{1,2}^E(x-y))\,dxdy
\sim\frac{3}2\,\frac{{\rm area}(\mathcal{D})}{\pi^3}\times \frac{\log E}{E}.
\end{eqnarray*}
\end{lemma}
\begin{lemma}\label{oddio1d}
As $E\to +\infty$, we have 
\begin{eqnarray*}
{\rm Var}(a_{5,E})&=&4\int_{\mathcal{D}}\int_{\mathcal D}\big(r^E(x-y)^2\widetilde r_{1,1}^E(x-y)^2+\widetilde r^E_{0,1}(x-y)^4\\
&-&4r^E(x-y)\widetilde r^E_{1,1}(x-y)\widetilde r_{0,1}^E(x-y)^2\big)\,dxdy
\sim\frac32\,\frac{{\rm area}(\mathcal{D})}{\pi^3}\times \frac{\log E}{E},\\
{\rm Cov}(a_{5,E},a_{6,E})&=& 4\int_{\mathcal{D}}\int_{\mathcal D}(r^E(x-y)^2\widetilde r_{2,2}^E(x-y)^2+\widetilde r^E_{0,2}(x-y)^4\\
&-&4r^E(x-y)\widetilde r_{0,2}^E(x-y)^2 \widetilde r^E_{2,2}(x-y))\,dxdy
\sim\frac{1}2\,\frac{{\rm area}(\mathcal{D})}{\pi^3}\times \frac{\log E}{E}.
\end{eqnarray*}
\end{lemma}
\begin{lemma}
As $E\to +\infty$, we have 
\begin{eqnarray*}
{\rm Var}(a_{6,E})&=&4\int_{\mathcal{D}}\int_{\mathcal D}\big(r^E(x-y)^2\widetilde r_{2,2}^E(x-y)^2+\widetilde r^E_{0,2}(x-y)^4\\
&-& 4r^E(x-y) \widetilde r^E_{2,2}(x-y)\widetilde r_{0,2}^E(x-y)^2\big)\,dxdy
\sim\frac32\,\frac{{\rm area}(\mathcal{D})}{\pi^3}\times \frac{\log E}{E}.
\end{eqnarray*}
\end{lemma}
\begin{lemma}\label{oddio2}
As $E\to +\infty$, we have
\begin{eqnarray*}
{\rm Var}(b_{1,E})
&=&4\int_{\mathcal{D}}\int_{\mathcal D}r^E(x-y)^4(u)\,dxdy
\sim\frac38\,\frac{{\rm area}(\mathcal{D})}{\pi^3}\times \frac{\log E}{E},\\
{\rm Cov}(b_{1,E},b_{2,E})
&=&4\int_{\mathcal{D}}\int_{\mathcal D}r^E(x-y)^2\widetilde r_{0,1}^E(x-y)^2\,dxdy
\sim\frac1{8}\,\frac{{\rm area}(\mathcal{D})}{\pi^3}\times \frac{\log E}{E},\\
{\rm Cov}(b_{1,E},b_{3,E})
&=&4\int_{\mathcal{D}}\int_{\mathcal D}r^E(x-y)^2\widetilde r_{0,2}^E(x-y)^2\,dxdy
\sim\frac1{8}\,\frac{{\rm area}(\mathcal{D})}{\pi^3}\times \frac{\log E}{E},\\
{\rm Cov}(b_{1,E},b_{4,E})
&=&4\int_{\mathcal{D}}\int_{\mathcal D}r^E(x-y)^2\widetilde r_{0,1}^E(x-y)^2\,dxdy
\sim\frac1{8}\,\frac{{\rm area}(\mathcal{D})}{\pi^3}\times \frac{\log E}{E}\\
{\rm Cov}(b_{1,E},b_{5,E})
&=&4\int_{\mathcal{D}}\int_{\mathcal D}r^E(x-y)^2\widetilde r_{0,2}^E(x-y)^2\,dxdy
\sim\frac1{8}\,\frac{{\rm area}(\mathcal{D})}{\pi^3}\times \frac{\log E}{E},\\
{\rm Cov}(b_{1,E},b_{6,E})
&=&4\int_{\mathcal{D}}\int_{\mathcal D}\widetilde r_{0,1}^E(x-y)^4\,dxdy
\sim\frac9{16}\,\frac{{\rm area}(\mathcal{D})}{\pi^3}\times \frac{\log E}{E},\\
{\rm Cov}(b_{1,E},b_{7,E})
&=&4\int_{\mathcal{D}}\int_{\mathcal D}\widetilde r_{0,2}^E(x-y)^4\,dxdy
\sim\frac9{16}\,\frac{{\rm area}(\mathcal{D})}{\pi^3}\times \frac{\log E}{E},\\
{\rm Cov}(b_{1,E},b_{8,E})
&=&4\int_{\mathcal{D}}\int_{\mathcal D}\widetilde r_{0,1}^E(x-y)^2\widetilde r_{0,2}^E(x-y)^2\,dxdy
\sim\frac3{16}\,\frac{{\rm area}(\mathcal{D})}{\pi^3}\times \frac{\log E}{E},\\
{\rm Cov}(b_{1,E},b_{9,E})
&=&4\int_{\mathcal{D}}\int_{\mathcal D}\widetilde r_{0,1}^E(x-y)^2\widetilde r_{0,2}^E(x-y)^2\,dxdy
\sim\frac3{16}\,\frac{{\rm area}(\mathcal{D})}{\pi^3}\times \frac{\log E}{E},\\
{\rm Cov}(b_{1,E},b_{10,E})
&=&4\int_{\mathcal{D}}\int_{\mathcal D}\widetilde r_{0,1}^E(x-y)^2\widetilde r_{0,2}^E(x-y)^2\,dxdy
\sim\frac3{16}\,\frac{{\rm area}(\mathcal{D})}{\pi^3}\times \frac{\log E}{E}. 
\end{eqnarray*}
\end{lemma}
\begin{lemma}\label{oddio2a}
As $E\to +\infty$, we have 
\begin{eqnarray*}
{\rm Var}(b_{2,E})
&=&4\int_{\mathcal{D}}\int_{\mathcal D}r^E(x-y)^2\widetilde r_{1,1}^E(x-y)^2\,dxdy
\sim\frac9{16}\,\frac{{\rm area}(\mathcal{D})}{\pi^3}\times \frac{\log E}{E},\\
{\rm Cov}(b_{2,E},b_{3,E})
&=&4\int_{\mathcal{D}}\int_{\mathcal D}r^E(x-y)^2\widetilde r_{1,2}^E(x-y)^2\,dxdy
\sim\frac3{16}\,\frac{{\rm area}(\mathcal{D})}{\pi^3}\times \frac{\log E}{E},\\
{\rm Cov}(b_{2,E},b_{4,E})
&=&4\int_{\mathcal{D}}\int_{\mathcal D}\widetilde r_{0,1}^E(x-y)^4\,dxdy
\sim\frac9{16}\,\frac{{\rm area}(\mathcal{D})}{\pi^3}\times \frac{\log E}{E},\\
{\rm Cov}(b_{2,E},b_{5,E})
&=&4\int_{\mathcal{D}}\int_{\mathcal D}\widetilde r_{0,1}^E(x-y)^2\widetilde r_{0,2}^E(x-y)^2\,dxdy
\sim\frac3{16}\,\frac{{\rm area}(\mathcal{D})}{\pi^3}\times \frac{\log E}{E},\\
{\rm Cov}(b_{2,E},b_{6,E})
&=&4\int_{\mathcal{D}}\int_{\mathcal D}\widetilde r_{0,1}^E(x-y)^2\widetilde r_{1,1}^E(x-y)^2\,dxdy
\sim\frac5{16}\,\frac{{\rm area}(\mathcal{D})}{\pi^3}\times \frac{\log E}{E},\\
{\rm Cov}(b_{2,E},b_{7,E})
&=&4\int_{\mathcal{D}}\int_{\mathcal D}\widetilde r_{0,2}^E(x-y)^2\widetilde r_{1,2}^E(x-y)^2\,dxdy
\sim\frac1{16}\,\frac{{\rm area}(\mathcal{D})}{\pi^3}\times \frac{\log E}{E},\\
{\rm Cov}(b_{2,E},b_{8,E})
&=&4\int_{\mathcal{D}}\int_{\mathcal D}\widetilde r_{0,1}^E(x-y)^2\widetilde r_{1,2}^E(x-y)^2\,dxdy
\sim\frac1{16}\,\frac{{\rm area}(\mathcal{D})}{\pi^3}\times \frac{\log E}{E}.
\end{eqnarray*}
\begin{eqnarray*}
{\rm Cov}(b_{2,E},b_{9,E})
&=&4\int_{\mathcal{D}}\int_{\mathcal D}\widetilde r_{0,2}^E(x-y)^2\widetilde r_{1,1}^E(x-y)^2\,dxdy
\sim\frac1{16}\,\frac{{\rm area}(\mathcal{D})}{\pi^3}\times \frac{\log E}{E},\\
{\rm Cov}(b_{2,E},b_{10,E})
&=&4\int_{\mathcal{D}}\int_{\mathcal D} \widetilde r^E_{0,1}(x-y) \widetilde r^E_{0,2}(x-y) \widetilde r^E_{1,1}(x-y) \widetilde r^E_{1,2}(x-y) \,dxdy\\
&\sim& \frac1{16}\,\frac{{\rm area}(\mathcal{D})}{\pi^3}\times \frac{\log E}{E}. 
\end{eqnarray*}
\end{lemma}
\begin{lemma}\label{oddio2b}
As $E\to +\infty$, we have 
\begin{eqnarray*}
{\rm Var}(b_{3,E})
&=& 4 \int_{\mathcal{D}}\int_{\mathcal D} r^E(x-y)^2\widetilde r_{2,2}^E(x-y)^2 \,dxdy
\sim\frac9{16}\,\frac{{\rm area}(\mathcal{D})}{\pi^3}\times \frac{\log E}{E},\\
{\rm Cov}(b_{3,E},b_{4,E})
&=&4\int_{\mathcal{D}}\int_{\mathcal D}\widetilde r_{0,1}^E(x-y)^2\widetilde r_{0,2}^E(x-y)^2\,dxdy
\sim\frac3{16}\,\frac{{\rm area}(\mathcal{D})}{\pi^3}\times \frac{\log E}{E},\\
{\rm Cov}(b_{3,E},b_{5,E})
&=&4\int_{\mathcal{D}}\int_{\mathcal D}\widetilde r_{0,2}^E(x-y)^4\,dxdy
\sim\frac9{16}\,\frac{{\rm area}(\mathcal{D})}{\pi^3}\times \frac{\log E}{E},\\
{\rm Cov}(b_{3,E},b_{6,E})
&=&4\int_{\mathcal{D}}\int_{\mathcal D}\widetilde r_{0,1}^E(x-y)^2\widetilde r_{1,2}^E(x-y)^2\,dxdy
\sim\frac1{16}\,\frac{{\rm area}(\mathcal{D})}{\pi^3}\times \frac{\log E}{E},\\
{\rm Cov}(b_{3,E},b_{7,E})
&=&4\int_{\mathcal{D}}\int_{\mathcal D}\widetilde r_{0,2}^E(x-y)^2\widetilde r_{2,2}^E(x-y)^2\,dxdy
\sim\frac5{16}\,\frac{{\rm area}(\mathcal{D})}{\pi^3}\times \frac{\log E}{E},\\
{\rm Cov}(b_{3,E},b_{8,E})
&=&4\int_{\mathcal{D}}\int_{\mathcal D}\widetilde r_{0,1}^E(x-y)^2\widetilde r_{2,2}^E(x-y)^2\,dxdy
\sim\frac1{16}\,\frac{{\rm area}(\mathcal{D})}{\pi^3}\times \frac{\log E}{E},\\
{\rm Cov}(b_{3,E},b_{9,E})
&=&4\int_{\mathcal{D}}\int_{\mathcal D}\widetilde r_{0,2}^E(x-y)^2\widetilde r_{1,2}^E(x-y)^2\,dxdy
\sim\frac1{16}\,\frac{{\rm area}(\mathcal{D})}{\pi^3}\times \frac{\log E}{E},\\
{\rm Cov}(b_{3,E},b_{10,E})
&=&4\int_{\mathcal{D}}\int_{\mathcal D}\widetilde r^E_{0,1}(x-y) \widetilde r^E_{0,2}(x-y) \widetilde r^E_{2,2}(x-y) \widetilde r^E_{1,2}(x-y) \,dxdy\\
&\sim& \frac1{16}\,\frac{{\rm area}(\mathcal{D})}{\pi^3}\times \frac{\log E}{E}.
\end{eqnarray*}
\end{lemma}
\begin{lemma}\label{oddio2c}
As $E\to +\infty$, we have 
\begin{eqnarray*}
{\rm Var}(b_{4,E})
&=&4\int_{\mathcal{D}}\int_{\mathcal D} r^E(x-y)^2\widetilde r_{1,1}^E(x-y)^2\,dxdy
\sim\frac9{16}\,\frac{{\rm area}(\mathcal{D})}{\pi^3}\times \frac{\log E}{E},\\
{\rm Cov}(b_{4,E},b_{5,E})
&=& 4\int_{\mathcal{D}}\int_{\mathcal D}r^E(x-y)^2\widetilde r_{1,2}^E(x-y)^2)\,dxdy
\sim\frac3{16}\,\frac{{\rm area}(\mathcal{D})}{\pi^3}\times \frac{\log E}{E},\\
{\rm Cov}(b_{4,E},b_{6,E})
&=& 4\int_{\mathcal{D}}\int_{\mathcal D}\widetilde r_{0,1}^E(x-y)^2\widetilde r_{1,1}^E(x-y)^2\,dxdy
\sim\frac5{16}\,\frac{{\rm area}(\mathcal{D})}{\pi^3}\times \frac{\log E}{E},\\
{\rm Cov}(b_{4,E},b_{7,E})
&=&4\int_{\mathcal{D}}\int_{\mathcal D}\widetilde r_{0,2}^E(x-y)^2\widetilde r_{1,2}^E(x-y)^2\,dxdy
\sim\frac1{16}\,\frac{{\rm area}(\mathcal{D})}{\pi^3}\times \frac{\log E}{E},\\
{\rm Cov}(b_{4,E},b_{8,E})
&=&4\int_{\mathcal{D}}\int_{\mathcal D}\widetilde r_{0,2}^E(x-y)^2\widetilde r_{1,1}^E(x-y)^2\,dxdy
\sim\frac1{16}\,\frac{{\rm area}(\mathcal{D})}{\pi^3}\times \frac{\log E}{E},\\
{\rm Cov}(b_{4,E},b_{9,E})
&=&4\int_{\mathcal{D}}\int_{\mathcal D}\widetilde r_{0,1}^E(x-y)^2\widetilde r_{1,2}^E(x-y)^2\,dxdy
\sim\frac1{16}\,\frac{{\rm area}(\mathcal{D})}{\pi^3}\times \frac{\log E}{E},\\
{\rm Cov}(b_{4,E},b_{10,E})
&=&4\int_{\mathcal{D}}\int_{\mathcal D}\widetilde r^E_{0,1}(x-y) \widetilde r^E_{0,2}(x-y) \widetilde r^E_{1,1}(x-y) \widetilde r^E_{1,2}(x-y)\,dxdy\\
&\sim& \frac1{16}\,\frac{{\rm area}(\mathcal{D})}{\pi^3}\times \frac{\log E}{E}.
\end{eqnarray*}
\end{lemma}
\begin{lemma}\label{oddio2d} 
As $E\to +\infty$, we have 
\begin{eqnarray*}
{\rm Var}(b_{5,E})
&=& 4\int_{\mathcal{D}}\int_{\mathcal D}\widetilde r_{1,1}^E(x-y)^4 \,dxdy
\sim\frac9{16}\,\frac{{\rm area}(\mathcal{D})}{\pi^3}\times \frac{\log E}{E},\\
{\rm Cov}(b_{5,E},b_{6,E})
&=&4\int_{\mathcal{D}}\int_{\mathcal D}\widetilde r_{0,1}^E(x-y)^2\widetilde r_{1,2}^E(x-y)^2\,dxdy
\sim\frac1{16}\,\frac{{\rm area}(\mathcal{D})}{\pi^3}\times \frac{\log E}{E},\\
{\rm Cov}(b_{5,E},b_{7,E})
&=& 4\int_{\mathcal{D}}\int_{\mathcal D}\widetilde r_{0,2}^E(x-y)^2\widetilde r_{2,2}^E(x-y)^2\,dxdy
\sim\frac5{16}\,\frac{{\rm area}(\mathcal{D})}{\pi^3}\times \frac{\log E}{E},\\
{\rm Cov}(b_{5,E},b_{8,E})
&=&4\int_{\mathcal{D}}\int_{\mathcal D}\widetilde r_{0,2}^E(x-y)^2\widetilde r_{1,2}^E(x-y)^2\,dxdy
\sim\frac1{16}\,\frac{{\rm area}(\mathcal{D})}{\pi^3}\times \frac{\log E}{E},\\
{\rm Cov}(b_{5,E},b_{9,E})
&=&4\int_{\mathcal{D}}\int_{\mathcal D}\widetilde r_{0,1}^E(x-y)^2\widetilde r_{2,2}^E(x-y)^2\,dxdy
\sim\frac1{16}\,\frac{{\rm area}(\mathcal{D})}{\pi^3}\times \frac{\log E}{E},\\
{\rm Cov}(b_{5,E},b_{10,E})
&=&4\int_{\mathcal{D}}\int_{\mathcal D}\widetilde r^E_{0,1}(x-y) \widetilde r^E_{0,2}(x-y) \widetilde r^E_{2,2}(x-y) \widetilde r^E_{1,2}(x-y)\,dxdy\\
&\sim& \frac1{16}\,\frac{{\rm area}(\mathcal{D})}{\pi^3}\times \frac{\log E}{E}.
\end{eqnarray*}
\end{lemma}
\begin{lemma}\label{oddio2e}
As $E\to +\infty$, we have 
\begin{eqnarray*}
{\rm Var}(b_{6,E})
&=&4\int_{\mathcal{D}}\int_{\mathcal D}\widetilde r_{1,1}^E(x-y)^4\,dxdy
\sim\frac{105}{64}\,\frac{{\rm area}(\mathcal{D})}{\pi^3}\times \frac{\log E}{E},\\
{\rm Cov}(b_{6,E},b_{7,E})
&=&4\int_{\mathcal{D}}\int_{\mathcal D}\widetilde r_{1,2}^E(x-y)^4\,dxdy
\sim\frac9{64}\,\frac{{\rm area}(\mathcal{D})}{\pi^3}\times \frac{\log E}{E},\\
{\rm Cov}(b_{6,E},b_{8,E})
&=&4\int_{\mathcal{D}}\int_{\mathcal D}\widetilde r_{1,1}^E(x-y)^2\widetilde r_{1,2}^E(x-y)^2\,dxdy
\sim\frac{15}{64}\,\frac{{\rm area}(\mathcal{D})}{\pi^3}\times \frac{\log E}{E},\\
{\rm Cov}(b_{6,E},b_{9,E})
&=&4 \int_{\mathcal{D}}\int_{\mathcal D}\widetilde r_{1,1}^E(x-y)^2\widetilde r_{1,2}^E(x-y)^2)\,dxdy
\sim\frac{15}{64}\,\frac{{\rm area}(\mathcal{D})}{\pi^3}\times \frac{\log E}{E}\\
{\rm Cov}(b_{6,E},b_{10,E})
&=& 4\int_{\mathcal{D}}\int_{\mathcal D}\widetilde r_{1,1}^E(x-y)^2\widetilde r_{1,2}^E(x-y)^2\,dxdy
\sim\frac{15}{64}\,\frac{{\rm area}(\mathcal{D})}{\pi^3}\times \frac{\log E}{E}
\end{eqnarray*}
\end{lemma}
\begin{lemma}\label{oddio2f}
As $E\to +\infty$, we have 
\begin{eqnarray*}
{\rm Var}(b_{7,E})
&=&4 \int_{\mathcal{D}}\int_{\mathcal D}\widetilde r_{2,2}^E(x-y)^4\,dxdy
\sim\frac{105}{64}\,\frac{{\rm area}(\mathcal{D})}{\pi^3}\times \frac{\log E}{E},\\
{\rm Cov}(b_{7,E},b_{8,E})
&=&4\int_{\mathcal{D}}\int_{\mathcal D}\widetilde r_{2,2}^E(x-y)^2\widetilde r_{1,2}^E(x-y)^2\,dxdy
\sim\frac{15}{64}\,\frac{{\rm area}(\mathcal{D})}{\pi^3}\times \frac{\log E}{E},\\
{\rm Cov}(b_{7,E},b_{9,E})
&=&4\int_{\mathcal{D}}\int_{\mathcal D}\widetilde r_{2,2}^E(x-y)^2\widetilde r_{1,2}^E(x-y)^2\,dxdy
\sim\frac{15}{64}\,\frac{{\rm area}(\mathcal{D})}{\pi^3}\times \frac{\log E}{E},\\
{\rm Cov}(b_{7,E},b_{10,E})
&=&4\int_{\mathcal{D}}\int_{\mathcal D}\widetilde r_{2,2}^E(x-y)^2\widetilde r_{1,2}^E(x-y)^2\,dxdy
\sim\frac{15}{64}\,\frac{{\rm area}(\mathcal{D})}{\pi^3}\times \frac{\log E}{E}.
\end{eqnarray*}
\end{lemma}
\begin{lemma}\label{oddio2g}
As $E\to +\infty$, we have 
\begin{eqnarray*}
{\rm Var}(b_{8,E})
&=&4\int_{\mathcal{D}}\int_{\mathcal D}\widetilde r_{1,1}^E(x-y)^2\widetilde r_{2,2}^E(x-y)^2\,dxdy
\sim\frac{9}{64}\,\frac{{\rm area}(\mathcal{D})}{\pi^3}\times \frac{\log E}{E},\\
{\rm Cov}(b_{8,E},b_{9,E})
&=&4\int_{\mathcal{D}}\int_{\mathcal D}\widetilde r_{1,2}^E(x-y)^4\,dxdy
\sim\frac{9}{64}\,\frac{{\rm area}(\mathcal{D})}{\pi^3}\times \frac{\log E}{E},\\
{\rm Cov}(b_{8,E},b_{10,E})
&=&4\int_{\mathcal{D}}\int_{\mathcal D}\widetilde r_{1,1}^E(x-y)\widetilde r_{2,2}^E(x-y)\widetilde r_{1,2}^E(x-y)^2\,dxdy
\sim\frac{9}{64}\,\frac{{\rm area}(\mathcal{D})}{\pi^3}\times \frac{\log E}{E}.
\end{eqnarray*}
\end{lemma}
\begin{lemma}\label{oddio2h}
As $E\to +\infty$, we have 
\begin{eqnarray*}
{\rm Var}(b_{9,E})
&=&4\int_{\mathcal{D}}\int_{\mathcal D}\widetilde r_{1,1}^E(x-y)^2\widetilde r_{2,2}^E(x-y)^2\,dxdy
\sim\frac{9}{64}\,\frac{{\rm area}(\mathcal{D})}{\pi^3}\times \frac{\log E}{E},\\
{\rm Cov}(b_{9,E},b_{10,E})
&=&4\int_{\mathcal{D}}\int_{\mathcal D}\widetilde r_{1,1}^E(x-y)\widetilde r_{2,2}^E(x-y)\widetilde r_{1,2}^E(x-y)^2\,dxdy
\sim\frac{9}{64}\,\frac{{\rm area}(\mathcal{D})}{\pi^3}\times \frac{\log E}{E},\\
{\rm Var}(b_{9,E})
&=&\int_{\mathcal{D}}\int_{\mathcal D}(\widetilde r_{1,1}^E(x-y)^2\widetilde r_{2,2}^E(x-y)^2+2\widetilde r_{1,1}^E(x-y)\widetilde r_{2,2}^E(x-y)\widetilde r_{1,2}^E(x-y)^2\\
&+& \widetilde r_{1,2}^E(x-y)^4)\,dxdy
\sim\frac{9}{64}\,\frac{{\rm area}(\mathcal{D})}{\pi^3}\times \frac{\log E}{E}.
\end{eqnarray*}
\end{lemma}

\section*{Appendix C}

\begin{proof}[Proof of Lemma \ref{smallolog}]
Reasoning as in the proof of Proposition \ref{propbessel}, we have 
\begin{equation}\label{great}
\begin{split}
\int_{\mathcal D} \int_{\mathcal D} \widetilde r^E_{k,l}(x-y)^6\, dx dy = & \text{area}(\mathcal D) \int_0^{\text{diam}(\mathcal D)} d\phi\, \phi \int_0^{2\pi} \widetilde r^E_{k,l}(\phi \cos \theta, \phi \sin \theta)^6\, d\theta \cr
&+ O\left( \int_0^{\text{diam}(\mathcal D)} d\phi\, \phi^2  \int_0^{2\pi} \widetilde r^E_{k,l}(\phi \cos \theta, \phi \sin \theta)^6\, d\theta  \right).
\end{split}
\end{equation}
Performing the change of variable $\theta = \psi / \sqrt E$ in the first term on the r.h.s. of \paref{great} we obtain 
\begin{equation}\label{eqa}
\begin{split}
\text{area}&(\mathcal D)  \int_0^{\text{diam}(\mathcal D)} d\phi\, \phi \int_0^{2\pi} \widetilde r^E_{k,l}(\phi \cos \theta, \phi \sin \theta)^6\, d\theta\cr
& = \text{area}(\mathcal D)  \frac{1}{E} \int_0^{\sqrt E \cdot \text{diam}(\mathcal D)} d\psi\, \psi \int_0^{2\pi} \widetilde r^1_{k,l}(\psi \cos \theta, \psi \sin \theta)^6\, d\theta.
\end{split}
\end{equation}
Since $r^1(\psi \cos \theta, \psi \sin \theta)\to 1$, $\widetilde r^1_{0,i}(\psi \cos \theta, \psi \sin \theta)= O(\psi)$ and $\widetilde r^1_{i,i}(\psi \cos \theta, \psi \sin \theta)\to 1$,  $\widetilde r^1_{1,2}(\psi \cos \theta, \psi \sin \theta)=O(\psi^2)$ as $\psi \to 0$ uniformly on $\theta$ ($i=1,2$), then we can rewrite \paref{eqa} as 
\begin{equation}\label{eqb}
\begin{split}
\text{area}&(\mathcal D)  \frac{1}{E} \int_0^{\sqrt E \cdot \text{diam}(\mathcal D)} d\psi\, \psi \int_0^{2\pi} \widetilde r^1_{k,l}(\psi \cos \theta, \psi \sin \theta)^6\, d\theta \cr
&=  O\left ( \frac{1}{E} \right ) + \text{area}({\mathcal D}) \frac{1}{E}\int_1^{\sqrt E \cdot \text{diam}(\mathcal D)} d\psi\, \psi \int_0^{2\pi} \widetilde r^1_{k,l}(\psi \cos \theta, \psi \sin \theta)^6\, d\theta.
\end{split}
\end{equation}
Now using \paref{Ogrande} for the second term on the r.h.s. of \paref{eqb}, as $E\to +\infty$, we have 
\begin{equation}\label{eqc}
\begin{split}
\frac{1}{E}\int_1^{\sqrt E \cdot \text{diam}(\mathcal D)} d\psi\, \psi \int_0^{2\pi} \widetilde r^1_{k,l}(\psi \cos \theta, \psi \sin \theta)^6\, d\theta &\ll \frac{1}{E}\int_1^{\sqrt E \cdot \text{diam}(\mathcal D)} \frac{d\psi}{\psi^2}\cr
& \sim \frac{1}{E}.
\end{split}
\end{equation}
For the error term on the r.h.s. of \paref{great} an analogous argument yields, as $E\to +\infty$, 
\begin{equation}\label{eqd}
\int_0^{\text{diam}(\mathcal D)} d\phi\, \phi^2  \int_0^{2\pi} \widetilde r^E_{k,l}(\phi \cos \theta, \phi \sin \theta)^6\, d\theta \asymp \frac{\log E}{E \sqrt E}. 
\end{equation}
Thanks to \paref{eqc} and \paref{eqd}, \paref{great} concludes the proof. 

\end{proof}

From \paref{Omega}, the covariance matrix of $\nabla B_E(x)$ conditioned to $B_E(x)=B_E(0)=0$ is 
\begin{equation*}
\Omega_E(x)= 2\pi^2 E\, I_2 - \frac{\nabla r^E(x)^t\, \nabla r^E(x)}{1 - r^E(x)^2} ,
\end{equation*}
and its determinant is 
\begin{equation*}
\text{det}(\Omega^E(x)) = 2\pi^2 E \left(2\pi^2 E - \frac{\| \nabla r^E(x)\|^2}{1 - r^E(x)^2}    \right).
\end{equation*}
\begin{lemma}\label{lemtaylor}
As $x\to 0$, it holds 
\begin{equation*}
\Psi_E(x) := \frac{|{\rm{det}}(\Omega_E(x))|}{1 - r^E(x)^2} = \frac18 (2\pi^2 E)^2 + E^3 O(\|x\|^2),
\end{equation*}
where the constant involved in the ``O"-notation does not depend on $E$. 
\end{lemma}
\begin{proof}
The Taylor development of $r^E$ centered at $0$ is 
\begin{equation}\label{taylor k}
\begin{split}
r^E(x) 
&= 1 - 2\pi^2 E\,\frac{\|x\|^2}{2} + \frac{(2\pi^2E)^2 \|x\|^4}{16} + E^3 O(\| x\|^6),
\end{split}
\end{equation}
where, from now until the end of the proof, the constants involved in the ``O"-notation do not depend on $E$. 
From (\ref{taylor k}) it is immediate that 
\begin{equation}\label{taylor kquadro}
\begin{split}
1-r^E(x)^2
&=  2\pi^2 E\,\|x\|^2 - \frac{3}{8}(2\pi^2E)^2 \|x\|^4 + E^3 O(\| x\|^6).
\end{split}
\end{equation}
Analogously, we find that the Taylor development for $\| \nabla r^E(x)\|^2$ centered at $0$ is 
\begin{equation}\label{taylor norma quadra}
\begin{split}
\|\nabla r^E(x)\|^2 
&=  2\pi^2 E\left(2\pi^2 E \|x\|^2 + (2\pi^2E)^2 \frac{\|x\|^4}{2}   +  E^3 O(\| x\|^6)\right).
\end{split}
\end{equation}
From (\ref{taylor kquadro}) and (\ref{taylor norma quadra}) we get
\begin{equation}\label{taylor ratio}
\begin{split}
\frac{\|\nabla r^E(x)\|^2 }{1-r^E(x)^2} &= \frac{ 2\pi^2 E\left(2\pi^2 E \|x\|^2 + (2\pi^2E)^2 \frac{\|x\|^4}{2}   +  E^3 O(\| x\|^6)\right)}{2\pi^2 E\,\|x\|^2 - \frac{3}{8}(2\pi^2E)^2 \|x\|^4 + E^3 O(\| x\|^6)}\\
&= \frac{ (2\pi^2 E)^2 \|x\|^2 \left(1+ 2\pi^2E\frac{\|x\|^2}{2}   +  E^2 O(\| x\|^4)\right)}{2\pi^2 E\,\|x\|^2\left( 1 - 2\pi^2E \frac{3}{8}\|x\|^2 + E^2 O(\| x\|^4)\right)}\\
&= 2\pi^2 E \left(1+ 2\pi^2E\frac{\|x\|^2}{2}   +  E^2 O(\| x\|^4)\right ) \left (1 - 2\pi^2E\frac{3}{8} \|x\|^2 + E^2 O(\| x\|^4)\right )\\
&= 2\pi^2 E \left(1 - 2\pi^2 E \frac{1}{8} \|x\|^2   + E^2 O(\| x\|^4)   \right). 
\end{split}
\end{equation}
From (\ref{taylor ratio}) and using again (\ref{taylor kquadro}) we can write 
\begin{equation*}
\begin{split}
\Psi_E(x) &= \frac{\left |2\pi^2 E \left(2\pi^2 E - \frac{\| \nabla r^E(x)\|^2}{1 -  r^E(x)^2}   \right) \right |}{1 - k_E(x)^2} \\
&= \frac{(2\pi^2E)^3 \frac{1}{8} \|x\|^2   + E^4O(\| x\|^4)}{2\pi^2 E\,\|x\|^2 + E^2 O( \|x\|^4)}\\
&=\frac18 (2\pi^2 E)^2 \left( 1 + E O(\|x\|^2)\right)
\end{split}
\end{equation*}
which conclude the proof.

\end{proof}

\small{}

\end{document}